\documentclass[11pt]{article}
\usepackage{amsfonts}
\usepackage{amsmath}
\usepackage{amssymb}
\usepackage{amsthm}
\usepackage{verbatim}
\usepackage{hyperref}

\hypersetup{
  colorlinks,
  urlcolor=black,
  linkcolor=black,
  anchorcolor=black,
  citecolor=black,
}

\title{The diamond-free process}

\author{Michael E. Picollelli\footnote{
\small Department of Electrical \& Computer Engineering, University of Delaware, Newark, DE, USA.  E-mail: \texttt{mpicolle@udel.edu}}}

\date{}

\newcommand{\into}{\rightarrow}
\newcommand{\ep}{\varepsilon}
\newcommand{\bin}[2]{\binom{#1}{#2}}
\newcommand{\pr}[1]{\mbox{Pr}\left(#1\right)}
\newcommand{\ev}[1]{\mathbb{E}(#1)}
\newcommand{\lp}{\left(}
\newcommand{\rp}{\right)}
\newtheorem{theorem}{Theorem}

\newtheorem{lemma}{Lemma}
\newtheorem{corollary}{Corollary}
\newtheorem{claim}{Claim}

\newtheorem{conjecture}{Conjecture}
\newcommand{\ca}[1]{{\cal #1}}
\newcommand{\gd}{\ca{G}_i \land \ca{H}_i}
\newcommand{\wt}[1]{\widetilde{#1}}
\newcommand{\wtY}{\widetilde{Y}}

\begin{document}

\maketitle

\begin{abstract}
Let $K_4^-$ denote the diamond graph, formed by removing an edge from the complete graph $K_4$.  We consider the following random graph process: starting with $n$ isolated vertices, add edges uniformly at random provided no such edge creates a copy of $K_4^-$.  We show that, with probability tending to $1$ as $n\into \infty$, the final size of the graph produced is $\Theta(\sqrt{\log(n)}\cdot n^{3/2})$.   Our analysis also suggests that the graph produced after $i$ edges are added resembles the random graph, with the additional condition that the edges which do not lie on triangles form a random-looking subgraph.
\end{abstract}

\section{Introduction} \label{sec:intro}

The $H$-free process, where $H$ is a fixed graph, is the random graph process which begins with a graph $G_0$ on $n$ isolated vertices.  The graph $G_i$ is then formed by adding an edge $e_i$ selected uniformly at random from the pairs which neither form edges of $G_{i-1}$ nor create a copy of $H$ in $G_{i-1} + e_i$.  The process terminates with an $H$-free graph $G_M$ with $M=M(H)$ edges.

The earliest result on an $H$-free process is that of Ruci\'nski and Wormald \cite{RW}, who showed the maximum-degree $d$ process terminates in a graph with $\lfloor nd/2\rfloor$ edges (here $H$ is the star with $d+1$ leaves) with high probability.  (We say a sequence of events $A_n$ occurs \textbf{with high probability}, or simply \textbf{w.h.p.}, if $\lim_{n \into \infty} \pr{A_n} = 1$.) Erd\H os, Suen and Winkler considered the triangle-free and odd-cycle-free processes, showing that w.h.p. the odd-cycle-free process terminates with $\Theta(n^2)$ edges, and that w.h.p. $M(K_3)= \Omega(n^{3/2})$ and $M(K_3)= O(\log(n)\cdot n^{3/2})$, as well as bounds on the independence number of $G_{M(K_3)}$.

The more general $H$-free process was considered by Bollob\'as and Riordan \cite{BR} and by Osthus and Taraz \cite{OT}.  For a graph $G$, let $e(G)$ denote the number of edges and $v(G)$ the number of vertices.  We say a graph $H$ is \textbf{$2$-balanced} if \[\frac{e(H)-1}{v(H)-2} \ge \frac{e(F)-1}{v(F)-1}\] for all proper subgraphs $F$ of $H$ with $v(F) \ge 3$, and \textbf{strictly} $2$-balanced if the inequality is sharp for all such $F$.  Bollob\'as and Riordan produced lower bounds on the $H$-free process for strictly $2$-balanced $H$, as well as upper bounds matching the lower bounds up to a polylog factor for $H \in \{K_4,C_4\}$.  Osthus and Taraz then produced upper bounds matching, to a polylog factor, the lower bounds for the strictly $2$-balanced case.  Wolfovitz \cite{Wolf} considered the case of regular strictly $2$-balanced $H$, producing an improvement in the expected value of the final size that yielded new lower bounds on the Tur\'an numbers $ex(n,K_{r,r})$ for $r \ge 5$.

Finally, Bohman \cite{B} succeeded in showing w.h.p. $M(K_3) = \Theta(\sqrt{\log(n)}\cdot  n^{3/2})$, confirming a conjecture of Spencer \cite{Sp}, and that the independence number of $G_{M(K_3)}$ is of order $O(\sqrt{\log(n)}\cdot n^{1/2})$, through an application of the differential equations method (we refer the reader to \cite{W} for the general method and examples), and consequently produced lower bounds on the Ramsey number $R(3,t)$ matching Kim's celebrated result \cite{K}.  Bohman also produced a lower bound on $M(K_4)$ of order $\Omega(\sqrt[5]{\log(n)} n^{8/5})$ as well as bounds as the independence number, improving the best known lower bounds on $R(4,t)$.  Bohman and Keevash \cite{BK} then produced new lower bounds for the size of the $H$-free process for strictly $2$-balanced $H$, as well as bounds on the independence number for a range of $H$.  An immediate consequence of their work was an improvement on the asymptotic lower bounds for the Ramsey numbers $R(s,t)$, $s \ge 5$ and $t$ sufficiently large, as well as for the Tur\'an numbers $ex(n,K_{r,r})$, $r \ge 5$.

We emphasize that the majority of the work mentioned so far has treated the strictly $2$-balanced case.  We are aware of no nontrivial lower bounds (see Theorem $1$ of \cite{OT}) for the general case when $H$ is $2$-balanced.  For the diamond graph $K_4^-$, we have \[\frac{e(K_4^-)-1}{v(K_4^-)-2} = 2 = \frac{e(K_3)-1}{v(K_3)-2},\] so $K_4^-$ is $2$-balanced but not strictly so.  On the other hand, it is reasonable to suspect that the final size of the diamond-free process is of the same order of magnitude as that of the triangle-free process, and our main result is that this is the case.

\begin{theorem}\label{thm:finalsize}
With high probability, \[M(K_4^-) = \Theta(\sqrt{\log(n)}\cdot n^{3/2}).\]
\end{theorem}

\noindent Very recently, Warnke \cite{WA2} and Wolfovitz \cite{Wolf2} have independently established upper bounds on $M(K_4)$ that match Bohman's lower bound to within a constant factor, but for most graphs $H$ containing a cycle, determining the final asymptotic size of the $H$-free process remains open.

Our argument consists of adapting the approach of \cite{B} for the triangle-free process to the diamond-free process.  We note that a graph $G$ is diamond-free if and only if every edge of $G$ lies on at most one triangle.  It follows that the edges in the neighborhood of any given vertex form a matching, and consequently $\alpha(G) \ge \Delta(G)/2$: bounding the independence number yields a bound on the final size of the graph constructed by the process.

As an additional comparison to the triangle-free process, suppose that as we progress through the diamond-free process, we color the $i$th edge blue if its addition does not create a triangle, and green otherwise.  Let $G_{blue}$ and $G_{green}$ denote the blue and green subgraphs of $G_M$.  Clearly $G_{blue}$ is triangle-free, and $e(G_{blue}) \ge \frac{2}{3}e(G_M)$.  Our arguments imply the following bound on the independence number of $G_{blue}$:

\begin{theorem}\label{thm:blue}
There is a constant $\gamma>0$ such that, with high probability, $\alpha(G_{blue}) < \gamma \sqrt{\log(n)\cdot n}$.
\end{theorem}

\noindent Note that Theorem \ref{thm:blue} implies the upper bound in Theorem \ref{thm:finalsize}.\\

\subsection{Comparison with the Strictly Dense Case}

A fascinating idea underlying the analysis of the $H$-free process for strictly $2$-balanced $H$ is that the graph $G_i$ produced after $i$ edges are added resembles the random graph $G(n,i)$, chosen uniformly at random from all graphs on $n$ vertices with $i$ edges, for $i$ up to the point where the number of copies of $H$ in $G(n,i)$ is roughly the same as the number of edges.  For strictly $2$-balanced $H$, this is when $i \approx n^{2-(v(H)-1)/(e(H)-1)}$.  Lending support to this intuition, Bohman and Keevash \cite{BK} show that the number of copies of a given $H$-free graph in $G_i$ is roughly the same as in $G(n,i)$, for $i$ up to a multiple of $n^{2-(v(H)-1)/(e(H)-1)}\log(n)^{1/(e(H)-1)}$.  (Similar results for the triangle-free process were obtained by Wolfovitz \cite{Wolf3}).  Extending results of Gerke and Makai \cite{GM} for the triangle-free process, Warnke \cite{WA1} showed that there are constants $c_1,c_2$ depending only on $H$ so that, w.h.p., the graph $G_{M(H)}$ contains no subgraphs on at most $n^{c_1}$ vertices with density at least $c_2$.  As a result of our analysis, one can modify Warnke's approach to show the existence of constants $c_1$ and $c_2$ so that the same conclusion holds for $G_{M(K_4^-)}$; we leave these details to the interested reader.

We take the view that the graph $G_i$ produced by the diamond-free process resembles the random graph $G(n,i)$, or, to simplify our view, the binomial random graph $G(n,p)$, where $p=p(i) = 2i/n^2$, but with an additional property that the edges which do not lie on triangles of $G_i$ form a random-looking subgraph.  For $p = 2\sqrt{\log(n)}/\sqrt{n}$, for example, w.h.p. every pair of vertices in $G(n,p)$ has codegree at least $2$ and consequently every pair of vertices that is not an edge would create a diamond if added, suggesting our upper bound on $M(K_4^-)$.  On the other hand, if a pair of vertices $uv$ has codegree $0$, or has codegree $1$ but neither edge connecting $u$ and $v$ to their mutual neighbor lies on a triangle, then the addition of $uv$ to the graph would \textit{not} create a diamond.  The edges which do not lie on triangles play an important role in our analysis, but it is reasonable to suspect that they are uniformly distributed among the edges.

We mention that there is a more substantial difference between the diamond-free process and strictly $2$-balanced case.  For the general $H$-free process, we say a pair of vertices $uv$ in $G_i$ is \textbf{open} if it does not form an edge of $G_i$ and $G_i + uv$ is $H$-free.  If a pair $uv$ is not an edge of $G_i$ and is not open, we say it is \textbf{closed}.  Bohman and Keevash's analysis (see Lemma 6.1 and Corollary 6.2 in \cite{BK}) shows that in the early evolution of the process for strictly $2$-balanced $H$, the open pairs effectively look the same, in the following sense: if $uv,u'v'$ are distinct open pairs in $G_i$, then the probability that $uv$ becomes closed by the choice of $e_{i+1}$ is nearly the same as the probability $u'v'$ becomes closed, up to relative error terms that tend to $0$ as $n \into \infty$.  In the diamond-free process, as one might suspect, it will turn out that open pairs whose addition would create a triangle are more likely to become closed than those whose addition would not.  This difference will have a nontrivial effect on our continuous approximation for the number of open pairs in a given step.

\subsection{Random Variables} \label{sec:defs}

Let $[n] = \{1,2,\ldots,n\}$ be the vertex set of the process, and $G_i$ be the graph given by the first $i$ edges selected by the process.  $G_i$ partitions $\bin{[n]}{2}$ into three sets: $E_i,O_i,C_i$.  The set $E_i$ is simply the edge set of $G_i$.
We say a pair $uv \in \bin{[n]}{2}\setminus E_i$ is \textbf{open}, and $uv \in O_i$, if $G_i + uv$ is diamond-free; otherwise $uv$ is \textbf{closed} and $uv \in C_i$.

As any pair with codegree at least $2$ in $G_i$ must be closed, we partition $E_i$ into $E_{0,i} \cup E_{1,i}$ and $O_i$ into $O_{0,i} \cup O_{1,i}$, where pairs in $E_{j,i}$ and $O_{j,i}$ have codegree $j$ in $G_i$.  For $j \in \{0,1\}$ let $Q_j(i) = |O_{j,i}|$ and let $Q(i) = |O_i|$.

Viewing $K_4^-$ as having a ``central" edge of codegree $2$ and four ``outer" edges of codegrees $1$, it follows easily that $O_{0,i}$ is the set of all pairs in $\bin{[n]}{2}\setminus E_i$ of codegree $0$.  It also follows that a pair $uv \in \bin{[n]}{2}\setminus E_i$ of codegree $1$ lies in $O_{1,i}$ if and only if the edges connecting $u$ and $v$ to their mutual neighbor lie in $E_{0,i}$. Returning briefly to the blue-green coloring introduced in the previous section, we also note that edge $e_i$ is blue if and only if $e_i \in O_{0,i-1}$.  Thus, every edge in $E_{0,i}$ is blue, but every triangle in $G_i$ contains two blue edges and a single green edge.

We introduce the continuous time variable $t$ and relate it to the process by setting $t = t(i) = i/n^{3/2}$.  As mentioned above, we take the heuristic view that the graph $G_i$ produced by the process resembles the random graph $G_{n,p}$, where $p=p(t)=2t/\sqrt{n}$, with the exception that it contains no diamonds.  We view the edges that do not lie on triangles, $E_{0,i}$, as forming a random subgraph with $r(t)n^{3/2}$ edges.  Loosely speaking, we select each edge with probability $p(t)$, and from those edges we select those which do not lie on triangles with probability $r(t)/t$.

If we suppose that the variables are close to their expected values, this suggests that \[|O_{0,i}| \approx \lp 1 - \lp\frac{2t}{\sqrt{n}}\rp^2 \rp^{n-2} \bin{n}{2} \approx \frac{e^{-4t^2}}{2} \cdot n^2,\] and by the comments above, \[|O_{1,i}|\approx (n-2) \lp \frac{2r(t)}{\sqrt{n}} \rp^2 \lp 1 - \lp \frac{2t}{\sqrt{n}}\rp^2 \rp^{n-3} \bin{n}{2} \approx \frac{4r(t)^2 e^{-4t^2}}{2}\cdot n^2.\]

To determine $r=r(t)$, we suppose that it is differentiable: assuming $|E_{0,i}| \approx r(t)n^{3/2}$, this implies $|E_{0,i+1}| - |E_{0,i}| \approx r'(t)$.  Of course, when we add an edge from $O_{0,i}$, $|E_{0,i}|$ increases by $1$, while when we add an edge from $O_{1,i}$, $|E_{0,i}|$ decreases by $2$, so the expected one-step change in $|E_{0,i}|$ is approximately \[ \frac{Q_0 - 2 Q_1}{Q_0 + Q_1} \approx \frac{1 - 8r^2}{1 + 4r^2}.\] This suggests the differential equation
\begin{equation}\label{eq:oder}
\frac{dr}{dt} = \frac{1-8r^2}{1 + 4r^2}
\end{equation} with initial condition $r(0)=0$, which has the following implicit solution:
\begin{equation}\label{eq:solr}
8t + 4r(t) - 3\sqrt{2}\cdot\text{arctanh}(2\sqrt{2}\cdot r(t)) = 0.
\end{equation}

\noindent It follows from \eqref{eq:oder} and \eqref{eq:solr} that $r(t)$ is nondecreasing and $\lim_{t \into \infty} r(t) = 1/2\sqrt{2}$.  Presuming these approximations are valid, this suggests the diamond-free process terminates with $t = \Theta(\sqrt{\log(n)})$, which is precisely what we will show.

To justify these approximations, we introduce the following definitions: for pairs $uv \in \bin{[n]}{2}$, let
\begin{eqnarray*}
X_{uv}(i) &=& \{w \in [n]: uw,vw \in O_i\},\\
Y_{uv}(i) &=& \{w \in [n]: |\{uw,vw\} \cap E_i| = |\{uw,vw\} \cap O_i| = 1\}, \text{ and }\\
Z_{uv}(i) &=& \{w \in [n]: uw,vw \in E_i\}.
\end{eqnarray*}  We call vertices in $X_{uv}(i)$ \textbf{open} with respect to $uv$, vertices in $Y_{uv}(i)$ \textbf{partial} with respect to $uv$, and vertices in $Z_{uv}(i)$ \textbf{complete} with respect to $uv$.  Additionally, if $w \in Y_{uv}(i)$ and, say, $uw \in O_i$, we say the pair $uw$ is \textbf{partial} with respect to $uv$, and we write $\wt{Y}_{uv}(i)$ for the set of partial pairs.

Similar to the analysis of the triangle-free process, we will use the partial vertices to track the changes to the open pairs, and the open vertices to track the changes to the partial vertices.  To do so effectively, we will need to further partition $X_{uv}$ and $Y_{uv}$, and we will first give some motivation for that.

Suppose we have $G_i$ and $e_{i+1}$ is selected.  If $e_{i+1} \in O_{0,i}$, then the only open pairs which become closed connect $e_{i+1}$ to $Y_{e_{i+1}}(i)$ - but not every such pair becomes closed.  For $w \in Y_{e_{i+1}}(i)$, we have $w \in Y_{e_{i+1}}(i+1)$ if and only if the edge connecting $e_{i+1}$ to $w$ lies in $E_{0,i}$ and the open pair lies in $O_{0,i}$.  As the addition of $e_{i+1}$ did not create a triangle, both $e_{i+1}$ and the edge connecting $w$ to $e_{i+1}$ lie in $E_{0,i+1}$, while the open pair connecting $w$ to $e_{i+1}$ lies in $O_{1,i+1}$.

On the other hand, if $e_{i+1} = uv \in O_{1,i}$, then letting $z$ be the unique common neighbor of $u$ and $v$ in $G_i$, the open pairs which become closed are all of those which connect $e_{i+1}$ to $Y_{e_{i+1}}$, as well as all open pairs connecting $uz$ to $Y_{uz}(i)$ and $vz$ to $Y_{vz}(i)$: $uv$, $uz$, and $vz$ then lie in $E_{1,i+1}$.

Consequently, we define the following subsets of $X_{uv}$: for $uv \in \bin{[n]}{2}$, let
\begin{eqnarray*}
X_{0,uv}(i) &=& \{w \in X_{uv}(i) \mid uw,vw \in O_{0,i}\},\\
X_{1,uv}(i) &=& \{w \in X_{uv}(i) \mid |\{uw,vw\}\cap O_{0,i}| = 1\}, \text{ and }\\
X_{2,uv}(i) &=& \{w \in X_{uv}(i) \mid uw,vw \in O_{1,i} \text{  and  } Z_{uw}(i) \cap Z_{vw}(i) = \emptyset\}.
\end{eqnarray*}
We note that these sets do not necessarily partition $X_{uv}$. However, if $w \in X_{uv}(i)$ with $Z_{uw}(i) \cap Z_{vw}(i) \ne \emptyset$, then $u,v,w$ lie in the neighborhood of some vertex $z$: the addition of either $uw$ or $vw$ in a later step implies the other is closed, so those $w$ will never lie in $Y_{uv}$.

We partition $Y_{uv}(i)$ into sets $Y_{j,k,uv}(i)$ for $j,k \in \{0,1\}$ as follows: for $(j,k) \ne (0,0)$, let \[Y_{j,k,uv}(i) = \{w \in Y_{uv}(i) : |\{uw,vw\} \cap E_{j,i}| = 1 = |\{uw,vw\} \cap O_{k,i}|\},\] i.e. the codegree of the edge connecting $w$ to $uv$ is $j$ and of the open pair connecting $w$ to $uv$ is $k$.  Let $Y_{0,0,uv}(i)$ be the set of $w \in Y_{uv}(i)$ such that the edge connecting $uv$ to $w$, say $uw$, lies in $E_{0,i}$ while the open pair connecting $uv$ to $w$, in this case $vw$, satisfies $Z_{vw}(i) \setminus \{u,v\} = \emptyset$.

The apparent disparity in the definition of $Y_{0,0,uv}$ is to allow us to keep track of vertices partial to edges in $E_{0,i}$.  If $uv \notin E_i$, then $w \in Y_{0,0,uv}(i)$ if and only if the codegrees of both the edge and the open pair connecting $w$ to $uv$ are $0$, while if $uv \in E_{0,i}$, then $Y_{uv}(i) = Y_{0,0,uv}(i)$.  We also note that contribution to $Y_{0,0,uv}(i)$ can only come from $X_{0,uv}$.

We will only track the random variables $|X_{1,uv}|$, $|X_{2,uv}|$ and $|Y_{j,k,uv}|$, $(j,k) \ne (0,0)$, over pairs $uv \notin E_i$, and we will only track $|X_{0,uv}|$ and $|Y_{0,0,uv}|$ over pairs not in $E_{1,i}$.  Formally, we will set $X_{1,uv}(i) = X_{1,uv}(i-1)$ if $uv \in E_i$, etc., which ensures the sets we are considering are well-defined throughout the process.
Finally, as above, we will use the notation $\wt{Y}_{j,k,uv}(i)$ to refer to the set of partial pairs connecting $uv$ to $Y_{j,k,uv}(i)$.

\subsection{Trajectory Derivation} \label{sec:trajderiv}

Our results will follow from an application of the differential equations method.  The general idea of the method is as follows: suppose we have a collection of sequences of random variables, and we can express their one-step expected changes in terms of the variables themselves.  This yields a system of ordinary differential equations, and we argue that the variables are tightly concentrated around the trajectory given by the solution of the system.  We mention that the arguments in this section will also be primarily heuristic, but they will be rigorously verified in later sections.

The variables we will track are the $Q_j$, $|X_{j,uv}|$, and $|Y_{j,k,uv}|$.  We begin by supposing that they are well-approximated by smooth functions $q_j, x_j, y_{j,k}$ as follows: $Q_j(i) \approx q_j(t)n^2$, $|X_{j,uv}(i)| \approx x_j(t)n$, and $|Y_{j,k,uv}(i)| \approx y_{j,k}(t)\sqrt{n}$ for all $uv \notin E_i$, and the approximations hold for $|X_{0,uv}(i)|$ and $|Y_{0,0,uv}(i)|$ for all $uv \in E_{0,i}$ as well.  Let $q = q_0 + q_1$ and $y = \sum_{j,k} y_{j,k}$: it follows that $Q(i) \approx q(t)n^2$ and $|Y_{uv}(i)| \approx y(t)\sqrt{n}$ for pairs $uv \notin E_i$.

We also assume that for given distinct pairs $uv$ and $wz$, the intersection of $\wt{Y}_{uv}(i)$ and $\wt{Y}_{wz}(i)$ is sufficiently small that the probability that $e_{i+1}$ lies in both is negligible.  If our variables follow the expected trajectories closely enough, then as $t(i+1) = t(i) + \frac{1}{n^{3/2}}$, the conditional expected one-step changes should be \begin{eqnarray*}
\ev{Q_j(i+1)-Q_j(i)} &\approx& \frac{dq_j}{dt}\cdot \sqrt{n},\\
\ev{|X_{j,uv}(i+1)|-|X_{j,uv}(i)|} &\approx& \frac{dx_j}{dt}\cdot \frac{1}{\sqrt{n}}, \text{ and }\\
\ev{|Y_{j,k,uv}(i+1)|-|Y_{j,k,uv}(i)|} &\approx& \frac{dy_{j,k}}{dt}\cdot \frac{1}{n}.\\
\end{eqnarray*}

We begin with the $Q_j$: suppose $e_{i+1} = uv$.  If $uv \in O_{0,i}$, then the pairs in $\wt{Y}_{uv}(i)\setminus \wt{Y}_{0,0,uv}(i)$ become closed, and the pairs in $\wt{Y}_{0,0,uv}(i)$ move from $O_{0,i}$ to $O_{1,i+1}$.  If $uv \in O_{1,i}$, then, letting $z \in Z_{uv}(i)$, all pairs in $\wt{Y}_{uv}(i) \cup \wt{Y}_{uz}(i) \cup \wt{Y}_{vz}(i)$ become closed.  Consequently,
\begin{eqnarray*}
\ev{Q_0(i+1)-Q_0(i)} &\approx& -|Y_{0,0,uv}(i)| - |Y_{1,0,uv}(i)| \approx -(y_{0,0}(t) + y_{1,0}(t))\sqrt{n}, \text{ and }\\
\ev{Q_1(i+1)-Q_1(i)} &\approx &  -|Y_{0,1,uv}(i)| - |Y_{1,1,uv}(i)|+ \frac{Q_0(i)(y_{0,0}(t)\sqrt{n}) + Q_1(i)(- 2y_{0,0}(t)\sqrt{n})}{Q(i)}\\
&\approx &  \lp-y_{0,1}(t) - y_{1,1}(t) +  \frac{q_0(t)y_{0,0}(t) - 2q_1(t) y_{0,0}(t)}{q(t)}\rp \sqrt{n},
\end{eqnarray*}
suggesting the differential equations \[\frac{dq_0}{dt} = -y_{0,0}-y_{1,0} \ \ \ \text{ and } \ \ \ \frac{dq_1}{dt} = -y_{0,1} - y_{1,1} + \frac{q_0y_{0,0} -2 q_1y_{0,0}}{q_0 + q_1}.\]

Next, we turn to the open and partial vertices.  We explicitly derive the differential equations for $x_0$ and $y_{0,0}$: the remaining equations, which we will state afterwards, fall out along the same lines.  Suppose that $uv \notin E_{1,i}$ and let $w \in X_{0,uv}(i)$.  The probability that $w \notin X_{0,uv}(i+1)$ is roughly $\frac{2y(t)}{q(t)} \frac{1}{n^{3/2}}$, so \[\ev{|X_{0,uv}(i+1)| - |X_{0,uv}(i)|} \approx -|X_{0,uv}(i)|\cdot \frac{2y(t)}{q(t)} \cdot \frac{1}{n^{3/2}} \approx -\frac{2x_0(t)y(t)}{q(t)}\cdot \frac{1}{\sqrt{n}},\] yielding the differential equation $dx_0/dt = -2x_0y/q$.  Similarly, the probability that $w \in Y_{0,0,uv}(i+1)$ is approximately $\frac{2}{q(t)}\cdot \frac{1}{n^2}$.

Now, fix a $w' \in Y_{0,0,uv}(i)$, and suppose $uw' \in E_{0,i}$ and $vw' \in O_i$ (recalling that $vw' \in O_{0,i}$ if $uv \notin E_i$ but $vw' \in O_{1,i}$ if $uv \in E_{0,i}$).  Then the probability that $w' \notin Y_{0,0,uv}(i+1)$ is approximately $\frac{y(t) + y_{0,0}(t)}{q(t)} \cdot\frac{1}{n^{3/2}}$, and consequently
\begin{eqnarray*}
\ev{|Y_{0,0,uv}(i+1)| - |Y_{0,0,uv}(i)|} &\approx& |X_{0,uv}|\frac{2}{q(t)}\cdot \frac{1}{n^2} - |Y_{0,0,uv}(i)|\frac{y(t) + y_{0,0}(t)}{q(t)} \frac{1}{n^{3/2}}\\
&\approx& \lp\frac{2x_0(t)}{q(t)} - \frac{y_{0,0}(t)(y(t) + y_{0,0}(t))}{q(t)} \rp \frac{1}{n},
\end{eqnarray*} yielding the differential equation \[\frac{dy_{0,0}}{dt} = \frac{2x_0 - y_{0,0}(y + y_{0,0})}{q}.\]

Analogous arguments for the remaining variables produce the system
\begin{equation}\label{eq:sysodes}
\begin{split}
\frac{dq_0}{dt} &  = -y_{0,0}-y_{1,0}, \ \ \ \ \ \ \ \ \ \ \frac{dq_1}{dt}  = -y_{0,1}-y_{1,1} + \frac{(q_0 - 2q_1)y_{0,0}}{q},\\
\frac{dx_0}{dt}  = -\frac{2x_0 y}{q},&\ \ \  \ \ \  \frac{dx_1}{dt} = \frac{2x_0y_{0,0} - 2x_1(y+y_{0,0})}{q}, \ \ \ \ \ \ \frac{dx_2}{dt} = \frac{x_1y_{0,0} - 2x_2(y+2y_{0,0})}{q},\\
\frac{dy_{0,0}}{dt} = &\frac{2x_0 - y_{0,0}(y_{0,0} + y)}{q},\ \ \ \ \ \ \ \ \ \ \ \frac{dy_{0,1}}{dt} = \frac{x_1 + y_{0,0}^2 - y_{0,1}(y + 3y_{0,0})}{q},\\
\frac{dy_{1,0}}{dt} =& \frac{x_1 + y_{0,0}^2 - y_{1,0}y}{q},\ \ \ \ \ \text{ and }\ \ \ \frac{dy_{1,1}}{dt} = \frac{2x_2 + y_{0,1}y_{0,0}+ y_{1,0}y_{0,0} - y_{1,1}(y+2y_{0,0})}{q}.\\
\end{split}
\end{equation}

\noindent As $|O_{0,0}| = \bin{n}{2}$, $|O_{1,0}| = 0$, $|X_{0,uv}(0)| = n-2$, and $|X_{j,uv}(0)|=0$ and $|Y_{k,l,uv}(0)| = 0$ for all $uv \in \bin{[n]}{2}$, $j \in \{1,2\}$, and $k,l \in \{0,1\}$, we get the initial conditions
\begin{equation}\label{eq:sysic}
q_0(0)=\frac{1}{2},\ \  q_1(0)=0,\ \  x_0(0)=1,\ \  x_j(0)=0 \text{ and } y_{k,l}(0)=0 \text{ for } j \in \{1,2\},\  k,l \in \{0,1\}.
\end{equation}

\noindent Our earlier estimates using $G_{n,p}$ suggested that $q_0 = e^{-4t^2}/2$ and $q_1 = 4r(t)^2e^{-4t^2}/2$, and thus the probability a pair lies in $|O_{j,i}|$ is approximately $2q_j$.  Applying this approach to the $x_j$ and $y_{k,l}$, this suggests
\[|X_{0,uv}(i)| \approx (n-2)(2q_0)\cdot(2q_0),\] and therefore $x_0 = (2q_0)^2$.  Similarly,  \[|Y_{0,0,uv}(i)| \approx (n-2)\cdot 2 \cdot \lp \frac{2r(t)}{\sqrt{n}}\rp \cdot (2q_0),\] and consequently $y_{0,0} = 2(2r(t))(2q_0)$.  Analogous computations yield
\[x_1 = 2(2q_0)(2q_1), \ \text{ and }\ x_2 = (2q_1)^2,\] as well as
\[y_{0,1} = 2(2r(t))(2q_1),\ \ \ y_{1,0} = 2(2(t-r(t))(2q_0), \ \ \text{ and }\ \ \  y_{1,1} = 2(2(t-r(t))(2q_1).\]

\noindent One can verify that this does (incredibly!) yield the solution to the above system, which we state as a theorem:
\begin{theorem} \label{thm:odesol}
The system \eqref{eq:sysodes} with initial conditions \eqref{eq:sysic} has the unique solution
\begin{equation}\label{eq:soltoodes}
\begin{split}
q_0(t) &= e^{-4t^2}/2,\ \ \ \ \ \ \ \ \ \ \ q_1(t) = 4r(t)^2e^{-4t^2}/2,\\
x_0(t) = e^{-8t^2},& \ \ \ \ \ \ \ x_1(t) = 8r(t)^2e^{-8t^2},\ \ \ \ \ \ \ x_2(t) = 16r(t)^4e^{-8t^2},\\
 y_{0,0}(t) &= 4r(t)e^{-4t^2},\ \ \ \ \ \ \ \ \ \ \ \ \ \ \ \ \ \ \ \ \ y_{0,1}(t) = 16r(t)^3e^{-4t^2},\\
y_{1,0}(t) &= 4(t-r(t))e^{-4t^2}, \ \ \ \ \text{ and } \ \ \ \
y_{1,1}(t) = 16r(t)^2(t-r(t))e^{-4t^2},\\
\end{split}
\end{equation}
where $r(t)$ satisfies \eqref{eq:solr}.
\end{theorem}

\subsection{Results}\label{sec:introresults}

Now, we formally state our results: let $K$ be a sufficiently large constant, and let
\[\theta_y(t) = e^{K(t^2+t)},\ \ \ \ \theta_q(t) = 1 + \int_0^t K\theta_y(\tau)\ d\tau,\ \ \ \ \ \theta_x(t) = e^{-4t^2}\theta_y(t),\ \text{  and  } \ \theta_r(t)=e^{4t^2}\theta_y(t).\]
Define $\delta_{\kappa} = \theta_{\kappa}\cdot n^{-1/6}$ for $\kappa=y,q,x,r$.  Let $0 < \ep < \frac{1}{40}$, and let $\mu=\mu(\ep,K)>0$ be sufficiently small so that $\theta_y\lp \mu\sqrt{\log(n)}\rp < n^{\ep}$ for all $n$ sufficiently large.  (We note that $\mu = \frac{\ep}{2K}$ suffices.)  Finally, let $m = \mu \sqrt{\log(n)}\cdot n^{3/2}$.\\

First, we show that, for $0 \le i \le m$, the estimates derived in the previous section are valid:

\begin{theorem} \label{thm:trajectory}
With high probability, for all $i$,  $i=0,\ldots,m$,
\begin{eqnarray*}
Q_j(i) &=& (q_j(t) \pm \delta_q(t))n^2 \ \ \text{ for } j \in \{0,1\},\\
|X_{0,uv}(i)| &=& (x_0(t) \pm \delta_x(t))n\ \  \text{ for }uv \notin E_{1,i},\\
|X_{j,uv}(i)| &=& (x_j(t) \pm \delta_x(t))n\ \ \text{ for } j>0 \text{ and }uv \notin E_{i},\\
|Y_{0,0,uv}(i)| &=& (y_{0,0}(t) \pm \delta_y(t))\sqrt{n}\ \  \text{ for }uv \notin E_{1,i},\text{ and }\\
|Y_{j,k,uv}(i)| &=& (y_{j,k}(t) \pm \delta_y(t))\sqrt{n}\ \  \text{ for } j,k \in \{0,1\}, (j,k) \ne (0,0) \text{ and } uv \notin E_{i}.\\
\end{eqnarray*}
\end{theorem}

\noindent As $e^{4t^2} < \theta_y(t) < n^{\ep}$ and $\ep < 1/12$, we have $q(t) \ge \frac{1}{2n^{\ep}}$ and $2\delta_q(t) \le \frac{2n^{\ep}}{n^{1/6}} < q_0(t)/2$ for $n$ sufficiently large and $0 \le i \le m$.  This implies $Q(m)>0$ and thus the lower bound of Theorem \ref{thm:finalsize}.\\

We next show that, with respect to the edges and the open pairs, $G_i$ remains approximately regular throughout:  for $v \in [n]$ and $j\in \{0,1\}$, let \[W_{j,v}(i) = \{w \in [n] : vw \in O_{j,i}\}\ \ \ \text{ and }\ \ \ N_{j,v}(i) = \{w \in [n] : vw \in E_{j,i}\},\] and let $d_{j,v}(i) = |N_{j,v}(i)|$.

\begin{theorem}\label{thm:regularity}
With high probability, for all $v \in [n]$, $j \in \{0,1\}$ and $i=0,\ldots,m$, we have
\begin{eqnarray*}
|W_{j,v}(i)| &=& (2q_j(t)\pm \delta_q(t))n,\\
d_{0,v}(i) &=& (2r(t) \pm \delta_r(t))n^{1/2}, \text{ and}\\
d_{1,v}(i) &=& (2(t-r(t)) \pm \delta_r(t))n^{1/2}.
\end{eqnarray*}
\end{theorem}

\noindent It follows from Theorem \ref{thm:regularity} that w.h.p., $|E_{0,i}| = (r(t) \pm \frac{\delta_r}{2})n^{3/2}$ for $i=0,\ldots,m$, and we note a second important consequence:

\begin{corollary}\label{cor:maxdeg}
With high probability, $\Delta(G_m) < 2(\mu\sqrt{\log(n)}+\delta_r)n^{1/2}$.
\end{corollary}

\noindent Next, let $G_{blue,i}$ be the blue subgraph of $G_i$.  Theorem \ref{thm:blue} follows from the next result:

\begin{theorem}\label{thm:indepno}
There is a constant $\gamma>0$ such that, with high probability, $\alpha(G_{blue,m}) < \gamma \sqrt{\log(n)\cdot n}$.
\end{theorem}

The remainder of this paper is organized as follows: Section \ref{sec:prelim} contains a few technical preliminaries, including a lemma from \cite{BK} which will allow us to justify our differential equations approximations.  In Section \ref{sec:trajectory}, we prove Theorem \ref{thm:trajectory}.  In Section \ref{sec:regularity}, we prove Theorem \ref{thm:regularity}.  Finally, in Section \ref{sec:independence}, we prove Theorem \ref{thm:indepno}.

\section{Preliminaries}\label{sec:prelim}

\subsection{The Differential Equations Method}\label{sec:demeth}

We use the notation ``$\pm$" in two distinct ways throughout this paper.  The notation $a \pm b$ will be taken to mean the interval $\{a + xb : -1 \le x \le 1\}$; distinct instances of $\pm$ used this way in the same expression will be treated independently, i.e. $(a \pm b)(c \pm d)$ will be taken to mean $\{(b + x_1c)(d + x_2e) : -1 \le x_1,x_2 \le 1\}$.  We will also write $a = b \pm c$ instead of $a \in b \pm c$.

For a sequence of random variables $A(1),A(2),\ldots,$, we will use $A^{\pm}$ to denote pairs of sequences of nonnegative random variables $A^+(1),A^+(2),\ldots$ and $A^-(1),A^-(2),\ldots$, such that \[A(i+1)-A(i) = A^+(i) - A^-(i).\]  Similarly, for a differentiable function $f(t)$, we will use $f^+$ and $f^-$ to denote the positive and negative parts of $f'(t)$.

To show that the variables follow the conjectured trajectories, we appeal to an approach to the differential equations method presented in Lemma 7.3 from \cite{BK}.  Aside from notation changes ($U_{j,A}$ in place of $X_{j,A}$, $U_{j,A}^{\pm}$ in place of $Y_{j,A}^{\pm}$), the only difference between the statement here and in \cite{BK} is that we have set an error parameter $e_j=0$; our arguments will not rely on it.

We reproduce from \cite{BK} the setup for the necessary lemma: suppose we have a stochastic graph process defined on $[n]$, where $n$ is large.  Let $l$ be a fixed positive integer, and for $j \in [l]$, let $k_j, S_j$ be parameters (which can depend on $n$).

Now, suppose for each $j \in [l]$ and $A \in \bin{[n]}{k_j}$, there is a sequence of random variables $U_{j,A}(i)$, defined for $i=0,\ldots,m$ and measurable with respect to the underlying graph process.

Further, we suppose \[U_{j,A}(i+1) - U_{j,A}(i) = U^{+}_{j,A}(i) - U^{-}_{j,A}(i),\] where $U^{+}_{j,A}(i),U^{-}_{j,A}(i) \ge 0$.
We relate these sequences to functions on $[0,\infty)$ by letting $t=i/s$ for some function $s=s(n)$ that tends to infinity.  The goal is then to argue that, for some collection $u_j(t)$ of continuous functions, \[U_{j,A}(i) \approx u_j(t)S_j\] for all $j \in [l]$ and $A \in \bin{[n]}{k_j}$, $i=0\ldots,m$.  We view $1 \le j \le l$ as the type of random variable, and the set $A$ as giving its position in the graph.  The parameter $S_j$ is the size-scaling for the variable.

\begin{lemma}[\cite{BK}, Lemma 7.3]\label{lem:demeth}
Let $0 < \epsilon < 1$ and $c,C> 0$ be constants, and suppose for each $j \in [l]$ we have a parameter $s_j(n)$ and functions $u_j(t),\theta_j(t),\gamma_j(t)$ that are smooth and nonnegative for $t \ge 0$.  For $i^* = 1,2,\ldots,m$, let $\ca{G}_{i^*}$ be the event that \[U_{j,A}(i) = \lp u_j(t) \pm \frac{\theta_j(t)}{s_j}\rp S_j\] for all $1 \le i \le i^*$, $1 \le j \le l$, and $A \in \bin{[n]}{k_j}$.  Suppose there is also a decreasing sequence of events $\ca{H}_i$, $1 \le i \le m$, such that $\lim_{n \into \infty} \pr{\ca{H}_m \mid \ca{G}_m} = 1$, and that the following conditions hold:

\begin{itemize}
    \item[1.] (Trend Hypothesis) When conditioning on $\ca{G}_i \land \ca{H}_i$, we have \[\ev{U^{\pm}_{j,A}} = \lp u^{\pm}_j(t) \pm \frac{h_j(t)}{4s_j}\rp\frac{S_j}{s},\]
        for all $j \in [l]$ and $A \in \bin{[n]}{k_j}$, where $u^{\pm}_j(t)$ and $h_j(t)$ are smooth nonnegative functions such that \[u_j'(t) = u_j^+(t) - u_j^-(t) \ \ \text{ and }\ \  h_j(t) = (\gamma_j)'(t);\]
    \item[2.] (Boundedness Hypothesis) For each $j \in [l]$, conditional on $\ca{G}_i \land \ca{H}_i$, we have \[U_{j,A}^{\pm}(i) < \frac{S_j}{s_j^2 k_j n^{\epsilon}};\]
    \item[3.] (Initial Conditions) For all $j \in [l]$, we have $\gamma_j(0) = 0$ and $U_{j,A}(0)= S_ju_j(0)$ for all $A \in \bin{[n]}{k_j}$;
    \item[4.] We have $n^{3\epsilon} < s < m < n^2$, $m \le n^{\ep/2}s$, $s \ge 40Cs_j^2 k_j n^{\epsilon}$, $n^{2\epsilon} \le s_j < n^{-\epsilon} s$,
        \[\inf_{t \ge 0} \theta_j(t) - \gamma_j(t)/2 > c,\]
        \[\sup_{t \ge 0} |u_j^{\pm}(t)| < C, \ \ \ \sup_{t \ge 0} |u_j'(t)| < C, \ \ \ \ \int_0^{\infty} |u_j''(t)|\ dt < C,\]
        \[ \sup_{0 \le t \le m/s} |h_j(t)| < n^{\epsilon}, \ \ \ \int_0^{m/s} |h_j'(t)|\ dt < n^{\epsilon}.\]
\end{itemize}
Then $Pr[\ca{G}_m \land \ca{H}_m] \into 1$ as $n \into \infty$.\\

\end{lemma}

We mention three technical issues which will arise from our application of Lemma \ref{lem:demeth}, all of which are easily dealt with.  The first is that we will apply Lemma \ref{lem:demeth} to prove Theorems \ref{thm:trajectory} and \ref{thm:regularity}, as well as Lemma \ref{lem:wopen} in Section \ref{sec:independence}.  In our proof of Theorem \ref{thm:regularity}, we include the conclusion of Theorem \ref{thm:trajectory} up to step $i$ in our high probability event $\ca{H}_i$, while our proof of Lemma \ref{lem:wopen} will include the corresponding conclusions of both Theorems \ref{thm:trajectory} and \ref{thm:regularity} in $\ca{H}_i$.   To avoid needing to verify that $\lim_{n \into \infty} \pr{\ca{H}_m | \ca{G}_m} = 1$ each time, we formally view the three applications of Lemma \ref{lem:demeth} as a single large application.  However, for ease of reading we prefer to separate the three arguments and present them individually.

The second issue is that a few of our variables do not have initial error $0$.  We will instead point out in these cases that the initial error is $o(S_j/s_j)$, and view our arguments as applied to the variables $U_{j,A}(i) - (U_{j,A}(0) - u_j(0)S_j)$: as our error terms are of the form $\Omega(1)S_j/s_j$, this subtlety can mostly be ignored.  We mention that one could easily modify the proof of Lemma 7.3 from \cite{BK} such that the conclusion holds with the weaker assumptions that the initial error is $o(S_j/s_j)$ and $\lim_{n \into \infty} \pr{\ca{H}_m} = 1$.

The final issue is that we only track some of our variables over a restricted collection of pairs.  Recall that, say, we only track $|X_{2,A}(i)|$ over pairs $A \notin E_i$, whereas the variables in the lemma are tracked over all of $\bin{[n]}{2}$.  We handle this by modifying the definition of the variables as follows: for the variable $U_{j,A}(i)$ corresponding to $|X_{2,A}(i)|$, if $A \notin E_i$ we set $U_{j,A}(i)=|X_{2,A}(i)|$, while if $A \in E_i$, we set $U_{j,A}^+(i-1) = u_j^+(t(i-1))S_j/s$ and $U_{j,A}^-(i-1)=u_j^-(t(i-1))S_j/s$, and $U_{j,uv}(i) = U_{j,A}(i-1) + U^+_{j,A}(i-1) - U^-_{j,A}(i-1)$.

Consequently, if $A \in E_i$, the trend hypothesis follows trivially, and the boundedness hypothesis follows from the inequalities $|u_j^{\pm}| \le C$ and $s \ge 40Cs_j^2k_jn^{\ep}$.  If, on the other hand, $A \notin E_i$, we appeal to the fact that the trend hypothesis asserts that $\ev{U_{j,A}^{\pm}(i) | \gd}$ lies in an interval centered at $u_j^{\pm}S_j/s$.  As $\ev{U_{j,A}^{\pm}(i) | \gd}$ is a convex combination of $\ev{U_{j,A}^{\pm}(i) | \gd \land (A \notin E_{i+1})}$ and $\ev{U_{j,A}^{\pm}(i) | \gd \land (A \in E_{i+1})}=u_j^{\pm}S_j/s$, it suffices to show the bounds of the trend and boundedness hypotheses hold, conditioned on $\gd \land (A \notin E_{i+1})$.

\subsection{Inequalities and Additional Lemmas}

Most of our arguments will focus on applying Lemma \ref{lem:demeth}, the most technical part of which is verifying the trend hypothesis for each of the variables.  The aim of this section is to compile a collection of simple observations which facilitate the later computations.

First, from Theorem \ref{thm:odesol}, the definitions of $\theta_y, \theta_x$, $\theta_q$, and recalling that $0 \le r(t) \le 1/2\sqrt{2}$, it follows that for $t \in [0,\mu\sqrt{\log(n)}]$, the following inequalities hold:

\[q_1 \le q_0,\ \ \ \  x_j \le x_0=4q_0^2,\ \ \ \  y_{0,0} \le 2\sqrt{2}q_0,\ \ \ \ y(t) \le 12tq_0\le \frac{3}{2},\]
\[ \delta_y^2 = o(\delta_x), \ \ \ \delta_y^2=o(\delta_q),\ \ \ \delta_x = 2q_0\delta_y,\ \ \ \ 2\delta_q < q_0/2, \ \ \text{ and } \ \  t\delta_q \le \delta_y.\]

From these inequalities, we can easily show the following lemma:

\begin{lemma}\label{lem:approx}
For $t \in [0,\mu\sqrt{\log(n)}]$, assuming all functions below are evaluated at $t$ and $n$ is sufficiently large,
\begin{eqnarray*}
\frac{(q_a \pm \delta_q)(y_{k,l} \pm \delta_y)}{q \pm 2\delta_q} &\subseteq&   \frac{q_a\cdot y_{k,l}}{q} \pm 74\delta_y,\\
\frac{(x_j \pm \delta_x) (y_{k,l} \pm \delta_y)}{q \pm 2\delta_q} &\subseteq& \frac{x_j\cdot y_{k,l}}{q} \pm (120t + 5)\delta_x,\\
\frac{x_j \pm \delta_x}{q \pm 2\delta_q}\ \ \ \ \  &\subseteq& \frac{x_j}{q} \pm 24 \delta_y,\ \text{ and }\\
\frac{(y_{k,l} \pm \delta_y) (y_{k',l'} \pm \delta_y)}{q \pm 2\delta_q} &\subseteq& \frac{y_{k,l}\cdot y_{k',l'}}{q} \pm ( 624t + 1)\delta_y.\\
\end{eqnarray*}
\end{lemma}

\begin{proof}
As the same algebraic manipulations are employed to prove all four containments, we explicitly show the second one:
\begin{eqnarray*}
\frac{(x_j \pm \delta_x) (y_{k,l} \pm \delta_y)}{q \pm 2\delta_q} &\subseteq& \frac{x_jy_{k,l}}{q \pm 2\delta_q} \pm \lp \frac{\delta_x y_{k,l} + \delta_y x_j + \delta_x\delta_y}{q_0/2} \rp\\
&\subseteq& \frac{x_jy_{k,l}}{q} \pm \lp \frac{x_jy_{k,l}(2\delta_q)}{q(q - 2\delta_q)} + \frac{12tq_0\delta_x + 4q_0^2\delta_y + 2q_0\delta_y^2}{q_0/2} \rp\\
&\subseteq& \frac{x_jy_{k,l}}{q} \pm \lp \frac{(4q_0^2) \cdot 12tq_0(2\delta_y)}{q_0(q_0/2)} + 24t\delta_x + 8q_0\delta_y + 4\delta_y^2 \rp\\
&\subseteq& \frac{x_jy_{k,l}}{q} \pm \lp 96t(2q_0\delta_y) + 24t\delta_x + 4\delta_x + \delta_x \rp\\
&=&  \frac{x_jy_{k,l}}{q} \pm (120t + 5)\delta_x.
\end{eqnarray*}
\end{proof}

While the exact values of the constants in the right-hand side of Lemma \ref{lem:approx} are themselves not essential to the argument, what is essential is that they are independent of the actual value of $K$.  This implies that replacing instances of $q_a, x_j, y_{k,l}$ with $(q_a \pm \delta_q),(x_j \pm \delta_x),(y_{k,l} \pm \delta_y)$ in the positive and negative parts of $\frac{dx_b}{dt}$, for example, produces error terms of the form $O(t+1)\delta_x$, where the constants hidden in the $O()$ notation are independent of $K$.  Similarly, the same substitutions in the positive and negative parts of $\frac{dy_{k,l}}{dt}$ produce error terms of the form $O(t+1)\delta_y$, and in $\frac{dq_a}{dt}$ the error terms have the form $O(1)\delta_y$.  Our verification of the trend hypothesis will essentially follow from this observation by choosing $K$ sufficiently large.

Finally, as in the triangle-free process we will ultimately show that $|Z_{uv}(i)|$ is suitably bounded throughout the diamond-free process, specifically that $|Z_{uv}(i)| \le \log^2(n)$ for $0\le i \le m$.  This bounding of the codegrees allows us to argue that relatively few open pairs are partial to many distinct pairs in $G_i$:

\begin{lemma} \label{lem:bdchange} Conditioned on $|Z_{uv}(i)| \le \log^2(n)$ for all $uv \in \bin{[n]}{2}$, if $A \subseteq \bin{[n]}{2} \setminus E_{1,i}$, the number of open pairs which are partial to at least two distinct pairs in $A$ is at most $\bin{|A|}{2}\log^2(n)$.
\end{lemma}

\begin{proof}
If $uv,u'v' \in \bin{[n]}{2} \setminus E_{1,i}$ are disjoint, there can be at most two pairs which are partial with respect to both $uv$ and $u'v'$.  If, on the other hand, $uv,vw \in \bin{[n]}{2} \setminus E_{1,i}$ are distinct, the number of pairs which are partial to both is at most $|Z_{uw}(i)|$.
\end{proof}

\section{The Lower Bound - Proof of Theorem \ref{thm:trajectory}}\label{sec:trajectory}

To prove Theorem \ref{thm:trajectory}, we apply Lemma \ref{lem:demeth} with the following values:  $s=n^{3/2}$, $\epsilon=1/12$, $m$ is as given in the introduction, $c=1/2$, and $C$ is a suitably large constant.  Additionally, $s_j = n^{1/6}$ and $\gamma_j = \theta_j-1$ for all $j$.

\begin{itemize}
\item For $j \in \{1,2\}$, we set $k_j=n$, $S_j=n^2$, $U_{j,[n]}=Q_{j-1}$, $u_j = q_{j-1}$, and $\theta_j=\theta_q$.
\item For $j \in \{3,4,5\}$, we set $k_j=2$, $S_j = n$, $U_{j,A} = |X_{j-3,A}|$, $u_j = x_{j-3}$, and $\theta_j=\theta_x$.
\item For $j \in \{6,\ldots,9\}$, we set $k_j = 2$, $S_j = \sqrt{n}$ and $\theta_j = \theta_y$.  We let $U_{6,A},\ldots,U_{9,A}$ and $u_6,\ldots,u_9$ be $|Y_{0,0,A}|,|Y_{0,1,A}|,\ldots,|Y_{1,1,A}|$ and $y_{0,0},\ldots,y_{1,1}$, respectively.
\item We let $\ca{H}_{i^*}$ be the event that $|Z_{uv}(i)| \le \log^2(n)$ for all $uv \in \bin{[n]}{2}$, $1 \le i \le i^*$.
\end{itemize}

Finally, we recall that we are only tracking $|Y_{0,0,uv}|$ and $|X_{0,uv}|$ over $uv \notin E_{1,i}$ and the remaining $|X_{j,uv}|$ and $|Y_{k,l,uv}|$ over $uv \notin E_i$.  As mentioned in the discussion in Section \ref{sec:demeth}, formally, for, say, $j=5$ (corresponding to $|X_{2,A}(i)|$), if $A \in E_{i}$ we set $U^+_{5,A}(i-1) = x_2^+(t(i-1))n/n^{3/2}$, $U^-_{5,A}(i-1) = x_2^-(t(i-1))n/n^{3/2}$, and
\[U_{5,A}(i) = \begin{cases}
|X_{0,A}(i)| & \text{ if } A \notin E_{i}\\
U_{5,A}(i-1) + U^+_{5,A}(i-1) - U^+_{5,A}(i-1) & \text{ if } A \in E_{i}.
\end{cases}\]

\noindent And, as mentioned, we will verify the trend and boundedness hypotheses by conditioning on the event $\ca{G}_i \land \ca{H}_i \land (A \notin E_{i+1})$. Throughout the remainder of this section we will implicitly assume the appropriate stopping-time modification for each of the affected variables, but write, say, $|X_{2,A}|$ and $X_{2,A}^{\pm}$ in place of $U_{5,A}$ and $U_{5,A}^{\pm}$ for ease of reading.

We note that, for each $j$, $h_j(t) = \gamma_j'(t) = \theta_j'(t)$, so
\begin{equation}\label{eq:hvals}
\frac{h_j(t)}{4s_j} = \begin{cases}
\frac{K}{4}\delta_y & \text{ if } j \in \{1,2\}\\
\frac{2(K-4)t+K}{4}\delta_x & \text{ if } j \in \{3,4,5\}, \text{ and }\\
\frac{2Kt + K}{4}\delta_y & \text{ if } j \in \{6,7,8,9\}.
\end{cases}
\end{equation}  As was done following the proof of Lemma \ref{lem:approx}, the constants hidden in the $O()$ notation throughout this section will be independent of $K$: from \eqref{eq:hvals} it will follow that some $K$ suitably large suffices.

As $Q_0(0)=\bin{n}{2} = q_0(0)n^2 - \frac{n}{2}$ and, for all $uv \in \bin{[n]}{2}$, $|X_{0,uv}(0)| = n-2 = x_0(0)n - 2$, it follows from $n/2 = o(n^{11/6})$ and $2 = o(n^{5/6})$ that the initial conditions of Lemma \ref{lem:demeth} are met.  The remaining sections are devoted to showing the remaining conditions holds.  In Sections \ref{sec:opedges}-\ref{sec:partverts}, we verify the trend and boundedness hypotheses of Lemma \ref{lem:demeth}.  In Section \ref{sec:cverts} we verify the condition $\lim_{n \into \infty} Pr[\ca{H}_m |\ca{G}_m] = 1$.  Finally, in Section \ref{sec:analytic}, we verify Condition 4 of Lemma \ref{lem:demeth} regarding the analytic requirements of our functions.\\

\subsection{Open Edges}\label{sec:opedges}

Here we verify the trend and boundedness hypotheses for $Q_0,Q_1$.  As mentioned previously, we write $Q_j(i+1)-Q_j(i) = Q^+_j(i) - Q^-_j(i)$, where $Q^+_j(i),Q^-_j(i) \ge 0$.  Provided doing so is unambiguous, we will simply write $Q_0$ in place of $Q_0(i)$, $q_0$ in place of $q_0(t(i))$, etc..

We begin with $Q_0$:  clearly $Q^+_0(i)=0$, while \[Q^-_0(i) = \begin{cases}
 |Y_{0,0,e_{i+1}}| + |Y_{1,0,e_{i+1}}| +1 & \text{ if } e_{i+1} \in O_{0,i}\\
 |Y_{0,0,e_{i+1}}| + |Y_{1,0,e_{i+1}}| & \text{ if } e_{i+1} \in O_{1,i}\\
\end{cases}\]

Consequently, \begin{eqnarray*}
\ev{Q^-_0(i)|\gd} &=& \lp y_{0,0} \pm \delta_y \rp \sqrt{n} + \lp y_{1,0} \pm \delta_y \rp \sqrt{n} \pm 1\\
&\subseteq& \lp y_{0,0} + y_{1,0} \pm \lp 2\delta_y + \frac{1}{\sqrt{n}}\rp \rp\sqrt{n},\\
&\subseteq& \lp y_{0,0} + y_{1,0} \pm 3\delta_y \rp\sqrt{n},
\end{eqnarray*} as $1/\sqrt{n} \ll \delta_y$, so the trend hypothesis for $Q_0$ holds by \eqref{eq:hvals}.\\

Turning to $Q_1$, we see
\[Q^+_1(i) =
\begin{cases}
|Y_{0,0,e_{i+1}}|& \text{ if }e_{i+1} \in O_{0,i}\\
0 & \text{ otherwise.}
\end{cases}\]
Applying Lemma \ref{lem:approx},
\begin{eqnarray*}
\ev{Q^+_1(i)|\gd} &=& \frac{Q_0\cdot \lp y_{0,0} \pm \delta_y \rp \sqrt{n}}{Q}\\
&\subseteq& \lp \frac{(q_0 \pm \delta_q)(y_{0,0} \pm \delta_y)}{q \pm 2\delta_q} \rp \sqrt{n}\\
&\subseteq& \lp \frac{q_0y_{0,0}}{q} \pm O(1)\delta_y \rp\sqrt{n}.\\
\end{eqnarray*}

Next, we turn to  $Q^-_1(i)$: if $e_{i+1} \in O_{0,i}$, then $Q^-_1(i) = |Y_{0,1,e_{i+1}}| + |Y_{1,1,e_{i+1}}|$.  If $e_{i+1} \in O_{1,i}$, then, supposing $e_{i+1}=uv$ and letting $z$ be the unique vertex in $Z_{uv}(i)$,
\[O_{1,i}\setminus O_{1,i+1} = \widetilde{Y}_{0,1,uv} \cup \widetilde{Y}_{1,1,uv} \cup \widetilde{Y}_{0,0,uw} \cup \widetilde{Y}_{0,0,vw}.\]  As $\widetilde{Y}_{0,0,uw} \cap \widetilde{Y}_{0,0,vw} = \{uv\}$, $\widetilde{Y}_{0,0,uw} \cap \widetilde{Y}_{uv}=\emptyset$, and $\widetilde{Y}_{0,0,vw} \cap \widetilde{Y}_{uv}=\emptyset$,
it follows that $Q^-_1(i) = |Y_{0,1,uv}| + |Y_{1,1,uv}| + |Y_{0,0,uw}| + |Y_{0,0,uv}|-1$.

Therefore,
\begin{eqnarray*}
\ev{Q^-_1(i) | \gd}  &=&\frac{Q_0 \cdot\lp y_{0,1}+y_{1,1} \pm 2\delta_y\rp \sqrt{n} + Q_1\cdot \lp y_{0,1} + y_{1,1} + 2y_{0,0} \pm 4\delta_y \pm \frac{1}{\sqrt{n}} \rp\sqrt{n} }{Q}\\
&\subseteq&
\lp \frac{q(y_{0,1} + y_{1,1}) + 2q_1y_{0,0}}{q} \pm \lp O(1)\delta_y +\frac{1}{\sqrt{n}} \rp \rp \sqrt{n},\\
&\subseteq&
\lp y_{0,1} + y_{1,1}+\frac{2q_1y_{0,0}}{q} \pm  O(1)\delta_y \rp \sqrt{n}\\
\end{eqnarray*}
so the trend hypothesis holds for $Q_1$ by \eqref{eq:hvals}.\\

For the boundedness hypothesis, it follows from the same considerations above that the maximum change in any of the $Q^{\pm}_j$ is at most $(y + 2y_{0,0} + 6\delta_y + \frac{1}{\sqrt{n}})\sqrt{n}\le \frac{9}{2}\sqrt{n} < n^{1-1/3-\epsilon}$ for $n$ sufficiently large.\\

\subsection{Open Vertices}\label{sec:opverts}

As in the previous section we begin with the trend hypothesis and show the boundedness hypothesis afterwards.  Suppose first that $uv \notin E_{1,i}$, and let $|X_{0,uv}(i+1)|-|X_{0,uv}(i)|=X^+_{0,uv}(i) - X^-_{0,uv}(i)$. Trivially $X^+_{0,uv}(i)=0$, we focus on $\ev{X^-_{0,uv}(i)|\gd \land (uv \notin E_{1,i+1})}$.  Formalizing the argument given in the introduction, fix a $w \in X_{0,uv}(i)$.  Then $w \notin X_{0,uv}(i+1)$  if and only if
\[e_{i+1} \in \begin{cases}
\widetilde{Y}_{uw}(i) \cup \widetilde{Y}_{vw}(i) \cup \{uw,vw\} & \text{if } uv \notin E_{0,i}, \text{ and }\\
\lp \widetilde{Y}_{uw}(i) \cup \widetilde{Y}_{vw}(i) \cup \{uw,vw\} \rp \setminus \widetilde{Y}_{uv}(i) & \text{if } uv \in E_{0,i}.
\end{cases}\]  As $\ca{G}_i$ holds, conditioning on $uv \notin E_{1,i+1}$ forbids only $e_{i+1}=uv$ if $uv \in O_{1,i}$ and $e_{i+1} \in Y_{0,0,uv}(i)$ if $e_{i+1} \in E_{0,i}$.  It easily follows that this forbids fewer than $10\sqrt{n}$ choices, and thus $e_{i+1}$ is selected uniformly at random from \[Q(i) \pm 10\sqrt{n} = Q(i)\lp 1 \pm \frac{40\sqrt{n}}{n^{2-\ep}}\rp  \subseteq Q(i)\lp 1 \pm o(\delta_x) \rp\] open pairs
(recalling that $Q(i) \ge n^{2-\ep}/4$ by the inequalities $q(t) \ge 1/2n^{\ep}$ and $2\delta_q(t) < q_0(t)/2$).

Applying Lemmas \ref{lem:approx} and \ref{lem:bdchange},
\begin{eqnarray*}
\ev{X^-_{0,uv}(i)|\gd \land (uv\notin E_{1,i})} &=& \sum_{w \in X_{0,uv}(i)} \lp \frac{|Y_{uw}| + |Y_{vw}| \pm (3\log^2(n)+2)}{Q(1 \pm o(\delta_x))}\rp\\
&\subseteq& \lp \frac{(x_0 \pm \delta_x)n\cdot \lp 2\sum_{0 \le j,k \le 1} (y_{j,k} \pm \delta_y) \pm \frac{4\log^2(n)}{\sqrt{n}}\rp\sqrt{n}}{(1 \pm o(\delta_x))(q \pm 2\delta_q)n^2}\rp\\
&\subseteq& \lp \frac{2x_0y}{q} \pm \lp O(t+1)\delta_x + \frac{(x_0 + \delta_x)4\log^2(n)}{\sqrt{n}(q_0/2)}\rp\rp \frac{(1 \pm o(\delta_x))}{\sqrt{n}}\\
&\subseteq& \lp \frac{2x_0y}{q} \pm \lp O(t+1)\delta_x + \frac{48\log^2(n)}{\sqrt{n}}\rp \rp \frac{(1 \pm o(\delta_x))}{\sqrt{n}}\\
&\subseteq& \lp x_0^- \pm  O(t+1)\delta_x \rp \frac{1}{\sqrt{n}},\\
\end{eqnarray*} as $\frac{48\log^2(n)}{\sqrt{n}} < n^{-1/6} \le \delta_x$, verifying the trend hypothesis for $|X_{0,uv}|$.\\

Next, we turn to $X_1$: suppose $uv \notin E_i$, and let $|X_{1,uv}(i+1)|-|X_{1,uv}(i)| = X^+_{1,uv}(i) - X^-_{1,uv}(i)$.  As we now condition additionally on $uv \notin E_{i+1}$, which prevents only $e_{i+1}=uv$ if $uv \in O_i$, $e_{i+1}$ is again chosen uniformly at random from $Q(i)\cdot (1 \pm o(\delta_x))$ choices.  It follows from the definitions that $X_{1,uv}(i+1)\setminus X_{1,uv}(i) \subseteq X_{0,uv}(i)$: a vertex $w \in X_{0,uv}(i)$ lies in $X_{1,uv}(i+1)$ if and only if $e_{i+1}$ lies in the symmetric difference of $\widetilde{Y}_{0,0,uw}(i)$ and $\widetilde{Y}_{0,0,vw}(i)$.  Consequently,
\begin{eqnarray*}
\ev{X^+_{1,uv}(i)|\gd\land (uv \notin E_{i+1})} &=& \sum_{w \in X_{0,uv}} \frac{|Y_{0,0,uw}| + |Y_{0,0,vw}| \pm \log^2(n)}{Q(1 \pm o(\delta_x))}\\
&\subseteq& \frac{(x_0 \pm \delta_x)n\lp 2y_{0,0} \pm 2\delta_y \pm \frac{\log^2(n)}{\sqrt{n}}\rp \sqrt{n}}{(1 \pm o(\delta_x))(q\pm 2\delta_q)n^2}\\
&\subseteq& \lp x_1^+ \pm  O(t+1)\delta_x \rp \frac{1}{\sqrt{n}}.\\
\end{eqnarray*}

Next, let $w \in X_{1,uv}(i)$ and, supposing $vw \in O_{1,i}$, let $z$ be the unique element of $Z_{vw}(i)$.  Then $w \notin X_{1,uv}(i+1)$ if and only if
\[e_{i+1} \in \widetilde{Y}_{uw}(i) \cup \widetilde{Y}_{vw}(i) \cup \widetilde{Y}_{vz}(i) \cup \widetilde{Y}_{wz}(i)\cup \{uw,vw\}.\] Recalling that $\widetilde{Y}_{vz} = \widetilde{Y}_{0,0,vz}$ and $\widetilde{Y}_{wz} = \widetilde{Y}_{0,0,wz}$,
\begin{eqnarray*}
\ev{X^-_{1,uv}(i) | \gd \land (uv \notin E_{i+1})} &=& \sum_{w \in X_{1,uv}} \frac{|Y_{uw}| + |Y_{vw}| + |Y_{0,0,vz}| + |Y_{0,0,wz}| \pm 7\log^2(n) }{Q(1 \pm o(\delta_x))}\\
&\subseteq&  \frac{(x_1 \pm \delta_x)n\lp 2y + 2y_{0,0} \pm 10\delta_y \pm \frac{7\log^2(n)}{\sqrt{n}}\rp \sqrt{n} }{(1 \pm o(\delta_x))(q \pm 2\delta_q)n^2}\\
&\subseteq&  \lp x_1^- \pm O(t+1)\delta_x \rp \frac{1}{\sqrt{n}},\\
\end{eqnarray*}
and the trend hypothesis for $|X_{1,uv}|$ follows.

Finally, we turn to $X_2$: suppose $uv \notin E_i$ and $|X_{2,uv}(i+1)|-|X_{2,uv}(i)| = X^+_{2,uv}(i) - X^-_{2,uv}(i)$.  We first note that $X_{2,uv}(i+1)\setminus X_{2,uv}(i) \subseteq X_{1,uv}(i)$.  This follows from the fact that if $w \in X_{0,uv}(i)$ and the choice of $e_{i+1}$ results in both $uw,vw \in O_{1,i+1}$, then $e_{i+1}=wz$ for some $z \in Z_{uv}(i)$, and therefore $Z_{uw}(i+1) \cap Z_{vw}(i+1) \ne \emptyset$.

Therefore, fix a $w \in X_{1,uv}(i)$, and suppose $uw \in O_{0,i}$ and $vw \in O_{1,i}$.  Let $z$ be the unique vertex in $Z_{vw}(i)$.  Then $w \in X_{2,uv}(i+1)$ if and only if \[e_{i+1} \in \widetilde{Y}_{0,0,uw}(i) \setminus \lp \widetilde{Y}_{vw}(i) \cup \widetilde{Y}_{vz}(i) \cup \widetilde{Y}_{wz}(i) \cup \{uz,wz\} \rp. \]  Applying the same manipulations as above, therefore
\begin{eqnarray*}
\ev{X^+_{2,uv}(i) | \gd \land (uv\notin E_{i+1})} &=& \sum_{w \in X_{1,uv}} \frac{\lp y_{0,0} \pm \delta_y \pm \frac{4\log^2(n)}{\sqrt{n}} \rp \sqrt{n}}{Q(1 \pm o(\delta_x))}\\
&\subseteq& \lp x_2^+ \pm  O(t+1)\delta_x \rp \frac{1}{\sqrt{n}}.\\
\end{eqnarray*}

Now, fixing a $w \in X_{2,uv}(i)$ and letting $z \in Z_{uw}(i)$, $z' \in Z_{vw}(i)$, we see that $w \notin X_{2,uv}(i+1)$ if and only if \[e_{i+1} \in \{uw,vw\} \cup \widetilde{Y}_{uw}(i) \cup \widetilde{Y}_{uz}(i) \cup \widetilde{Y}_{wz}(i) \cup \widetilde{Y}_{vw}(i) \cup \widetilde{Y}_{vz'}(i) \cup \widetilde{Y}_{wz'}(i)\]
Appealing to Lemma \ref{lem:bdchange} and applying the same reasoning, it follows that
\begin{eqnarray*}
\ev{X^-_{2,uv}(i) | \gd\land (uv \notin E_{i+1})} &=& \lp \frac{ (x_2 \pm \delta_x)\lp 2y + 4y_{0,0} \pm 12\delta_y \pm \frac{16\log^2(n)}{\sqrt{n}}\rp }{(1 \pm o(\delta_x))(q \pm 2\delta_q)}  \rp \frac{1}{\sqrt{n}}\\
&\subseteq& \lp x_2^- \pm  O(t+1)\delta_x \rp \frac{1}{\sqrt{n}}
\end{eqnarray*} and the result follows from \eqref{eq:hvals}.

For the boundedness hypothesis, as at most $\frac{9}{2}\sqrt{n}$ open pairs are affected by the addition of $e_{i+1}$ (i.e. go from open to closed or $O_{0,i}$ to $O_{1,i+1}$), the maximum value of $|X^{\pm}_{j,uv}|$ is at most $\frac{9}{2}\sqrt{n} < n^{1-1/3-\epsilon}/2$.

\subsection{Partial Vertices}\label{sec:partverts}

We begin with  $|Y_{0,0,uv}|$: suppose $uv \notin E_{1,i}$, and $|Y_{0,0,uv}(i+1)|-|Y_{0,0,uv}(i)| = Y^+_{0,0,uv}(i) - Y^-_{0,0,uv}(i)$.  By the stopping-time modification we condition on $uv \notin E_{1,i+1}$, which, as noted previously, implies $e_{i+1}$ is selected uniformly at random from $Q(i)(1 \pm o(\delta_x)) \subseteq Q(i)(1 \pm o(\delta_y))$ open pairs.

Let $w \in X_{0,uv}(i)$.  Then $w \in Y_{0,0,uv}(i+1)$ if and only if $e_{i+1} \in \{uw,vw\}$, and consequently \[\ev{Y^+_{0,0,uv}(i) | \gd} = \frac{2|X_{0,uv}(i)|}{Q} \subseteq \lp \frac{2(x_0 \pm \delta_x)}{q \pm 2\delta_q}\rp \frac{1}{n} = \lp \frac{2x_0}{q} \pm O(1)\delta_y \rp \frac{1}{n}.
\]

Next, let $w \in Y_{0,0,uv}(i)$, and suppose $uw \in E_{0,i}$ and $vw \in O_i$.  Then $w \notin Y_{0,0,uv}(i+1)$ if and only if
\[e_{i+1} \in
\begin{cases}
(\widetilde{Y}_{uw} \cup \widetilde{Y}_{vw} \cup \{vw\})\setminus \{uv\}  & \text{ if } uv \notin E_{0,i},\\
\lp \widetilde{Y}_{uw} \cup \widetilde{Y}_{vw} \cup \{vw\} \rp\setminus \widetilde{Y}_{0,0,uv} & \text{ if } uv \in E_{0,i}.\\
\end{cases}\]
Therefore,
\begin{align*}
\mathbb{E}(Y^-_{0,0,uv} &| \gd \land (uv \notin E_{1,i+1})) \\
&= \lp \sum_{w \in Y_{0,0,uv}} \frac{(y_{0,0} \pm \delta_y) + \sum_{0\le j,k \le 1} (y_{j,k} \pm \delta_y) \pm \frac{3\log^2(n)+2}{\sqrt{n}} }{Q(1 \pm o(\delta_y))} \rp \sqrt{n}\\
&\subseteq \lp \frac{(y_{0,0} \pm \delta_y)\lp (y_{0,0} \pm \delta_y) + \sum_{0\le j,k \le 1} (y_{j,k} \pm \delta_y) \pm \frac{3\log^2(n)+2}{\sqrt{n}} \rp }{(1 \pm o(\delta_y))(q \pm 2\delta_q)} \rp \frac{1}{n}\\
&\subseteq \lp \frac{y_{0,0}(y + y_{0,0})}{q} \pm \lp O(t+1)\delta_y + \frac{(y + \delta_y)(4\log^2(n))}{(q_0/2)\sqrt{n}} \rp \rp \frac{(1 \pm o(\delta_y))}{n}\\
&\subseteq \lp y_{0,0}^- \pm \lp O(t+1)\delta_y + \frac{(24t+4e^{4t^2}\delta_y)(4\log^2(n))}{\sqrt{n}} \rp \rp \frac{(1 \pm o(\delta_y))}{n}\\
&\subseteq \lp y_{0,0}^- \pm \lp O(t+1)\delta_y + \frac{100\log^{5/2}(n)}{\sqrt{n}} \rp \rp \frac{(1 \pm o(\delta_y))}{n}\\
&\subseteq \lp y_{0,0}^- \pm  O(t+1)\delta_y \rp \frac{1}{n},
\end{align*} by noting that $4e^{4t^2}\delta_y < \sqrt{\log(n)}$ and, provided $\mu \le 1$, $t \le \sqrt{\log(n)}$.\\

For the remainder of this subsection, it is sufficient to consider only $uv \notin E_i$, and we note conditioning on $\gd \land (uv \notin E_{i+1})$ leaves $Q(i)(1 \pm o(\delta_y))$ equally likely choices for $e_{i+1}$.  We begin with $|Y_{0,1,uv}|$: let $|Y_{0,1,uv}(i+1)| - |Y_{0,1,uv}(i)| = Y^+_{0,1,uv}(i) - Y^-_{0,1,uv}(i)$.  Clearly contribution to $Y_{0,1,uv}(i+1)$ can only come from $X_{1,uv}(i)$ or $Y_{0,0,uv}(i)$, and a vertex in $X_{1,uv}(i)$ enters $Y_{0,1,uv}(i+1)$ with probability $1/(Q(1 \pm o(\delta_y))$.  We therefore focus on the probability a fixed $w \in Y_{0,0,uv}(i)$ enters $Y_{0,1,uv}(i+1)$: suppose $uw \in E_{0,i}$ and $vw \in O_{0,i}$.  For $w$ to lie in $Y_{0,1,uv}(i+1)$, we must have $vw \in O_{1,i+1}$ and $uw \in E_{0,i+1}$, which occurs if and only if $e_{i+1} \in \widetilde{Y}_{0,0,vw}(i) \setminus  \widetilde{Y}_{uw}$.

Therefore,
\begin{eqnarray*}
\ev{Y^+_{0,1,uv}(i)|\gd \land (uv\notin E_{i+1})} &=& \frac{|X_{0,uv}| + \sum_{w \in Y_{0,0,uv}} \lp (y_{0,0} \pm \delta_y)\sqrt{n} \pm \log^2(n) \rp}{Q(1 \pm o(\delta_y))}\\
&\subseteq& \lp \frac{(x_0 \pm \delta_x) + (y_{0,0} \pm \delta_y) \lp (y_{0,0} \pm \delta_y) \pm \frac{\log^2(n)}{\sqrt{n}} \rp}{(1 \pm o(\delta_y))(q \pm 2\delta_q)} \rp \frac{1}{n}\\
&\subseteq& \lp y_{0,1}^+ \pm O(t+1)\delta_y \rp \frac{1}{n}.\\
\end{eqnarray*}

Suppose now that $w \in Y_{0,1,uv}(i)$, and that $uw \in E_{0,i}, vw \in O_{1,i}$, and $z$ is the unique vertex in $Z_{vw}(i)$.  In order to have $w \notin Y_{0,1,uv}(i+1)$, we must have \[e_{i+1} \in (\wtY_{uw} \cup \wtY_{vw} \cup \wtY_{wz} \cup \wtY_{vz} \cup \{vw\}) \setminus \{uv\} = (\wtY_{vw} \cup \wtY_{0,0,uw} \cup \wtY_{0,0,wz} \cup \wtY_{0,0,vz} \cup \{vw\})\setminus \{uv\}.\]

Therefore,
\begin{align*}
\mathbb{E}(Y^-_{0,1,uv}(i) &| \gd \land (uv \notin E_{i+1})) \\
&= \lp \sum_{w \in Y_{0,1,uv}} \frac{(\sum_{0\le j,k \le 1} (y_{j,k} \pm \delta_y)) + 3(y_{0,0} \pm \delta_y) \pm \frac{7\log^2(n)}{\sqrt{n}}}{(1 \pm o(\delta_y))(q \pm 2\delta_q)} \rp \frac{1}{n^{3/2}}\\
&\subseteq \lp \frac{(y_{0,1} \pm \delta_y)\lp (\sum_{0\le j,k \le 1} (y_{j,k} \pm \delta_y)) + 3(y_{0,0} \pm \delta_y) \pm \frac{7\log^2(n)}{\sqrt{n}} \rp }{(1 \pm o(\delta_y))(q \pm 2\delta_q)} \rp \frac{1}{n}\\
&\subseteq \lp y_{0,1}^- \pm O(t+1)\delta_y \rp \frac{1}{n}.
\end{align*}

Next, let $|Y_{1,0,uv}(i+1)| - |Y_{1,0,uv}(i)| = Y^+_{1,0,uv}(i) - Y^-_{1,0,uv}(i)$.  Contribution to $Y_{1,0,uv}(i+1)$ can only come from $X_{1,uv}(i)$ or $Y_{0,0,uv}(i)$: a vertex from $X_{1,uv}(i)$ enters $Y_{1,0,uv}(i+1)$ with probability $1/(Q(1 \pm o(\delta_y)))$, so consider a vertex $w \in Y_{0,0,uv}(i)$.  Suppose that $uw \in E_{0,i}$ while $vw \in O_{0,i}$.  Then $w \in Y_{1,0,uv}(i+1)$ if and only if $e_{i+1} \in \wtY_{0,0,uw}\setminus \wtY_{vw}$, so
\begin{eqnarray*}
\ev{Y^+_{1,0,uv}(i)|\gd \land (uv\notin E_{i+1})} &=& \frac{|X_{1,uv}| + \sum_{w \in Y_{0,0,uv}} \lp (y_{0,0} \pm \delta_y)\sqrt{n} \pm \log^2(n)\rp}{Q(1 \pm o(\delta_y))}\\
&\subseteq& \lp y_{1,0}^+ \pm O(t+1)\delta_y \rp \frac{1}{n}.
\end{eqnarray*}

Next, let $w \in Y_{1,0,uv}(i)$, and suppose $uw \in E_{1,i}$ and $vw \in O_{0,i}$.  Then $w \notin Y_{1,0,uv}(i+1)$ if and only if $e_{i+1} \in (\wtY_{vw} \cup \{vw\})\setminus \{uv\}$.  Therefore,
\begin{eqnarray*}
\ev{Y^-_{1,0,uv}(i)|\gd \land (uv \notin E_{i+1})} &=& \sum_{w \in Y_{1,0,uv}} \frac{\sum_{0\le j,k \le 1} (y_{j,k} \pm \delta_y)\sqrt{n} \pm 1}{Q(1 \pm o(\delta_y))}\\
&\subseteq& \lp y_{1,0}^-  \pm O(t+1)\delta_y \rp \frac{1}{n}.
\end{eqnarray*}

Finally, let $|Y_{1,1,uv}(i+1)|-|Y_{1,1,uv}(i)| = Y^+_{1,1,uv}(i) - Y^-_{1,1,uv}(i)$.  We see first that contribution to $Y_{1,1,uv}(i+1)$ can only come from $X_{2,uv}(i)$, $Y_{0,1,uv}(i)$ or $Y_{1,0,uv}(i)$.  It is trivial to see that it cannot come from $X_{0,uv}(i) \cup X_{1,uv}(i)$, but we mention that it cannot come from $Y_{0,0,uv}(i)$ either.  To see this, suppose $w \in Y_{0,0,uv}(i)$ with $uw \in E_{0,i}$ and $vw \in O_{0,i}$. If the addition of $e_{i+1}$ results in both $Z_{uw}(i+1)\ne \emptyset$ and $Z_{vw}(i+1) \ne \emptyset$, then $e_{i+1} = wz$ for some $z \in Z_{uv}(i)$.  As then $z \in Z_{vw}(i+1)$ and $wz \in E_{1,i+1}$ ($u,w,z$ form a triangle), $vw \notin O_{1,i+1}$ and $w \notin Y_{1,1,uv}(i+1)$.

To determine the expected contribution, it follows easily that a fixed vertex in $X_{2,uv}(i)$ enters $Y_{1,1,uv}(i+1)$ with probability $2/(Q(1 \pm o(\delta_y)))$.  Suppose now that $w \in Y_{0,1,uv}(i)$, and that $uw \in E_{0,i}$, $vw \in O_{1,i}$, and $z \in Z_{vw}(i)$.  Then $w \in Y_{1,1,uv}(i+1)$ if and only if $e_{i+1} \in \wtY_{0,0,uw} \setminus (\wtY_{vw} \cup \wtY_{0,0,vz} \cup \wtY_{0,0,wz})$.  Finally, suppose $w' \in Y_{1,0,uv}(i)$, and that $uw' \in E_{1,i}$ and $vw' \in O_{0,i}$.  Then $w' \in Y_{1,1,uv}(i+1)$ if and only if $e_{i+1} \in \wtY_{0,0,vw'}$.

Combining these observations,
\begin{align*}
\mathbb{E}(Y^+_{1,1,uv}(i)& |\gd \land (uv\notin E_{i+1}))\\
&= \frac{ 2|X_{2,uv}| + \sum_{w \in Y_{0,1,uv}} \lp (y_{0,0} \pm \delta_y)\sqrt{n} \pm 3\log^2(n) \rp + \sum_{w' \in Y_{1,0,uv}} (y_{0,0} \pm \delta_y) \sqrt{n}}{Q(1 \pm o(\delta_y))}\\
&\subseteq \lp \frac{2x_0 + y_{0,1}y_{0,0} + y_{1,0}y_{0,0}}{q} \pm \lp O(t+1)\delta_y + \frac{(y+\delta_y)3\log^2(n)}{(q_0/2)\sqrt{n}} \rp \rp \frac{(1 \pm o(\delta_y))}{n}\\
&\subseteq \lp y_{1,1}^+ \pm O(t+1)\delta_y \rp \frac{1}{n}.\\
\end{align*}

Now, let $w \in Y_{1,1,uv}(i)$, and suppose $uw \in E_{1,i}$ and $vw \in O_{1,i}$.  Let $z$ be the unique vertex in $Z_{vw}(i)$.  Then $w\notin Y_{1,1,uv}(i+1)$ if and only if \[e_{i+1} \in \lp \wtY_{vw} \cup \wtY_{0,0,wz} \cup \wtY_{0,0,vz} \cup \{vw\}\rp \setminus \{uv\}.\]

Therefore
\begin{align*}
\mathbb{E}(Y^-_{1,1,uv}(i)|\gd &\land (uv\notin E_{i+1}))\\
&= \sum_{w \in Y_{1,1,uv}} \frac{\sum_{0\le j,k\le 1} (y_{j,k} \pm \delta_y)\sqrt{n} \pm 2(y_{0,0} \pm \delta_y)\sqrt{n} \pm (3\log^2(n) + 2)}{Q(1 \pm o(\delta_y))} \\
&\subseteq \lp y_{1,1}^- \pm O(t+1)\delta_y \rp \frac{1}{n}.
\end{align*}

\noindent Provided $K$ is suitably large, it follows from \eqref{eq:hvals} that the trend hypothesis holds for the $|Y_{j,k,uv}|$.\\

To argue boundedness, we first note that if $e_{i+1} \in O_{0,i}$, then the only pairs which either become closed or move from $O_{0,i}$ to $O_{1,i+1}$ are those in $\wtY_{e_{i+1}}$, and no edges move from $E_{0,i}$ to $E_{1,i+1}$.  On the other hand, if $e_{i+1} \in O_{1,i}$, say $e_{i+1} = uv$, then letting $z \in Z_{uv}$, the open pairs which become closed are those in $\wtY_{uv} \cup \wtY_{uz} \cup \wtY_{vz}$ and the edges which enter $E_{1,i+1}$ from $E_{0,i}$ are $uz,vz$.  Conditioned on $\ca{H}_i$, Lemma \ref{lem:bdchange} implies $|Y^-_{j,k,uv}| \le 3\log^2(n) + 2 < \frac{\sqrt{n}}{2n^{1/3 + \epsilon}}$.

Next, we observe that, conditioned on $\ca{H}_i$, the maximum contribution to any $Y_{j,k,uv}(i+1)$ from some $X_{l,uv}(i)$ in a single step is at most $1$, while the maximum contribution from a $Y_{j',k',uv}(i)$ is at most $3\log^2(n) + 2$: our earlier arguments show that $|Y^+_{j,k,uv}(i)| <  6\log^2(n) +5 \le \frac{\sqrt{n}}{2n^{1/3 + \epsilon}}$, and the boundedness hypothesis holds.

\subsection{Complete Vertices}\label{sec:cverts}

Here we show that $\lim_{n \into \infty} \pr{\ca{H}_m | \ca{G}_m} = 1$.  First, we observe that for any $uv \in \bin{[n]}{2}$, \[\pr{|Z_{uv}(i+1)| = |Z_{uv}(i)| + 1 | \ca{G}_i} \le \frac{(y(t) + 4\delta_y(t))\sqrt{n}}{(q-2\delta_q)n^2} \le \frac{24t + 1}{n^{3/2}} < \frac{25\mu\sqrt{\log(n)}}{n^{3/2}}\] for $n$ sufficiently large.

Consequently,
\[\pr{\overline{\ca{H}}_m | \ca{G}_m} \le \bin{n}{2} \bin{\mu \sqrt{\log(n)}\cdot n^{3/2}}{\log^2(n)} \lp \frac{25\mu\sqrt{\log(n)}}{n^{3/2}} \rp^{\log^2(n)} < e^{-\Omega(\log^2(n)\log\log(n))} = o(1),\] and $\lim_{n \into \infty} \pr{\ca{H}_m| \ca{G}_m} = 1$.\\

\subsection{Analytic Considerations} \label{sec:analytic}

We now verify the inequalities of Condition 4 of Lemma \ref{lem:demeth}.  Recall that $m=\mu \sqrt{\log(n)} n^{3/2}$, $s=n^{3/2}$, $s_j = n^{1/6}$ for all $j$, and we set $c=\frac{1}{2}$, and $\epsilon = 1/12$ in our application as well as selected $C$ to be suitably large.

It follows immediately that $n^{3(1/12)} < s < m < n^2$ and $n^{2/12} \le s_j < n^{-1/12}s$.  Noting that each $k_j \le n$, it also follows that $40Cs_j^2k_j n^{1/12} \le 40C n^{17/12} \ll n^{3/2}=s$.  Additionally, since $\gamma_j = \theta_j - 1$, \[\inf_{t \ge 0} \theta_j(t) - \gamma_j(t)/2 \ge \inf_{t \ge 0} \theta_j(t)/2 + 1/2 > 1/2.\]

We next show that the expressions in Condition $4$ of Lemma \ref{lem:demeth} for which $C$ must be an upper bound are bounded: as there are finitely many such expressions, it follows that some $C$ suitably large suffices.  Recall that we write $q^{\pm}_a$, $x^{\pm}_j$, $y^{\pm}_{k,l}$ for the positive and negative parts of $dq_a/dt$, $dx_j/dt$, and $dy_{k,l}/dt$, respectively.

To argue that $|q^{\pm}_a|$, $|x^{\pm}_j|$, and $|y^{\pm}_{k,l}|$ are bounded over the reals, it suffices to observe that
\[ q^{\pm}_a \le y, \ \ \ x^+_j \le \frac{2x_0y_{0,0}}{q},\ \ \ x^-_j \le \frac{2x_0(y + 2y_{0,0})}{q},\ \ \  y^+_{k,l} \le \frac{2x_0 + yy_{0,0}}{q},\ \ \text{ and } \ \ y^-_{k,l} \le \frac{3y^2}{q}.\]
As $y \le 6te^{-4t^2}$, $y_{0,0} \le \sqrt{2}e^{-4t^2}$, $q \ge \frac{1}{2}e^{-4t^2}$, and $x_0 = e^{-8t^2}$, it follows that each is bounded by a function of the form $g_1(t)e^{-4t^2}$, where $g_1$ is quadratic in $t$, which itself is trivially bounded over $[0,\infty)$ by some constant, call it $C_1$.  The bounds on $|q_a'|,|x_j'|,|y_{k,l}'|$ follow immediately by noting each is at most $2C_1$.\\

Next, we bound $\int_0^{\infty} |q_a''(t)|\ dt$, $\int_0^{\infty} |x_j''(t)|\ dt$, and $\int_0^{\infty} |y_{k,l}''(t)|\ dt$.  First, as $r(t) \in [0,\frac{1}{2\sqrt{2}})$ for all $t \ge 0$, it follows from \eqref{eq:oder} that $|r'(t)|$ is bounded on $[0,\infty)$, and as \[r''(t) = -\frac{24r(t)r'(t)}{(1 + 4r(t)^2)^2},\] therefore $|r''(t)|=-r''(t)$ is bounded as well.

It then follows from Theorem \ref{thm:odesol} that we can bound $\max_{a,j,k,l}\{|q_a''(t)|,|x_j''(t)|,|y_{k,l}''(t)|\}$ by a function of the form $g_2(t)e^{-4t^2}$, where $g_2$ is a polynomial in $t$.  We let $C_2 = \int_0^{\infty} g_2(t)e^{-4t^2}\ dt$.  Provided we take $C > \max\{2C_1,C_2\}$, the remaining inequalities requiring $C$ are satisfied.

We turn finally to the inequalities involving $h_j$: as $h_j=\theta_j'$, and $|h_j'| = |\theta_j''| = \theta_j''$, and both $\theta_j'$ and $\theta_j''$ are strictly increasing on $[0,\infty)$, it follows that \[\int_{0}^{m/s} |h_j'(t)|\ dt \le h_j(m/s) = \sup_{0 \le t \le m/s} |h_j(t)|.\]
From \eqref{eq:hvals}, we have that \[h_j(m/s) \le 2(K+4)\lp \mu\sqrt{\log(n)} + 1\rp \theta_j\lp\mu\sqrt{\log(n)}\rp \le 3K\sqrt{\log(n)}\cdot n^{\ep} \ll n^{1/12}\] by recalling that $\ep < 1/40$ and our choice of $\mu$ ensured $\theta_j(\mu\sqrt{\log(n)}) < n^{\ep}$ for all $j$.
This completes the verification of Condition 4 of Lemma \ref{lem:demeth}, and consequently the proof of Theorem \ref{thm:trajectory}.\\

\section{Degree Regularity - Proof of Theorem \ref{thm:regularity}} \label{sec:regularity}

To prove Theorem \ref{thm:regularity} we apply Lemma \ref{lem:demeth} with the following values:  $s=n^{3/2}$, $\epsilon=1/12$, $m$ is as given in the introduction, $c=1/2$, and $C$ is a suitably large constant.  Additionally, $k_j=1$ and $s_j = n^{1/6}$.

\begin{itemize}
\item For $j \in \{1,2\}$, we set $S_j=n$, $U_{j,v}=|W_{j-1,v}|$, $u_j = 2q_{j-1}$, $\theta_j = \theta_q$ and $\gamma_j = \theta_q-1$.
\item For $j=3$, we set $S_3 = \sqrt{n}$, $U_{3,v} = d_{0,v}$, $u_3 = 2r(t)$, $\theta_3 = \theta_r$ and $\gamma_3 = \theta_r-1$.
\item For $j=4$, we set $S_4 = \sqrt{n}$, $U_{4,v} = d_{1,v}$, $u_4 = 2(t-r(t))$, $\theta_4 = \theta_r$ and $\gamma_4 = \theta_r-1$.
\item Let $\ca{H}_{i^*}$ be the event that the conclusions of Theorem \ref{thm:trajectory} hold for the first $i^*$ steps, and that $|Z_{uv}(i)| < \log^2(n)$ for $0 \le i \le i^*$ and all $uv \in \bin{[n]}{2}$.
\end{itemize}

It then follows that
\begin{equation}\label{eq:hvals2}
\frac{h_j(t)}{4s_j} = \begin{cases}
\frac{K}{4}\delta_y & \text{ if } j \in \{1,2\}\\
\frac{2(K+4)t+K}{4}\delta_r & \text{ if } j \in \{3,4\}.\\
\end{cases}
\end{equation}

We first argue that the trend and boundedness hypotheses hold: fix a $v \in [n]$, and for $j \in \{0,1\}$ let \[|W_{j,v}(i+1)|-|W_{j,v}(i)| = W_{j,v}^+(i) - W_{j,v}^-(i), \ \ \text{ and }\ \  d_{j,v}(i+1)-d_{j,v}(i) = d_{j,v}^+(i) - d_{j,v}^-(i).\]  For a fixed $w \in W_{1,v}(i)$, let $z=z(w)$ be the unique vertex in $Z_{vw}(i)$.\\

It follows from the definitions as well as Lemma \ref{lem:approx} and the manipulations in Section \ref{sec:opedges} that
\begin{eqnarray*}
\ev{W_{0,v}^+(i)|\gd} &=& 0\\
\ev{W_{0,v}^-(i)|\gd} &=& \sum_{w \in W_{0,v}} \frac{|Y_{vw}(i)| + 1}{Q(i)} \subseteq \lp \frac{2q_0y}{q} \pm O(1)\delta_y \rp \sqrt{n}\\
\ev{W_{1,v}^+(i) | \gd } &=& \sum_{w \in W_{0,v}(i)} \frac{|Y_{0,0,vw}(i)|}{Q(i)} \subseteq \lp \frac{2q_0y_{0,0}}{q} \pm O(1)\delta_y\rp \sqrt{n}\\
\ev{W_{1,v}^-(i) | \gd} &=& \sum_{w \in W_{1,v}(i)} \frac{|Y_{wv}| + |Y_{0,0,vz}| + |Y_{0,0,wz}| - 1}{Q(i)} \subseteq \lp \frac{2q_1(y + 2y_{0,0})}{q} \pm O(1)\delta_y\rp \sqrt{n}.\\
\end{eqnarray*}

As $y = 8tq$,  $y_{0,0} + y_{1,0} = 8tq_0$, and $y_{0,1} + y_{1,1} = 8tq_1$, the trend hypothesis for the $|W_{j,v}|$ follows from \eqref{eq:hvals2}, provided $K$ is sufficiently large.  Boundedness follows from the earlier observation that the maximum change to the $Q_j$ is at most $\frac{9}{2}\sqrt{n} \ll n^{1-1/3-\epsilon}$.\\

We turn now to the $d_{j,v}$: clearly we have
\begin{eqnarray*} \ev{d_{0,v}^+(i)| \gd} =  \frac{|W_{0,v}|}{Q}
&=& \lp \frac{2(q_0 \pm \delta_q)}{q \pm 2\delta_q}\rp \frac{1}{n}\\
&\subseteq& \lp \frac{2q_0}{q} \pm \lp \frac{2q_0(2\delta_q)}{q(q-2\delta_q)} + \frac{2\delta_q}{q_0/2} \rp \rp \frac{1}{n}\\
&\subseteq& \lp \frac{2}{1 + 4r^2} \pm  12e^{4t^2}\delta_q \rp \frac{1}{n}\\
&\subseteq& \lp \frac{2}{1 + 4r^2} \pm  O(1)\delta_r \rp \frac{1}{n}.\\
\end{eqnarray*}
Similarly,
\begin{eqnarray*}
\ev{d_{0,v}^-(i) | \gd} &=& \sum_{w \in N_{0,v}(i)} \frac{|Y_{0,0,vw}|}{Q} \subseteq \lp \frac{2ry_{0,0}}{q} \pm O(1)\delta_r\rp \frac{1}{n},\\
\ev{d_{1,v}^+(i) | \gd} &=&  \frac{ |W_{1,v}| + \sum_{w \in N_{0,v}(i)} |Y_{0,0,vw}|}{Q} \subseteq \lp \frac{2q_1 + 2ry_{0,0}}{q} \pm O(1)\delta_r \rp \frac{1}{n}, \text{ and }\\
\ev{d_{1,v}^-(i)|\gd} &=& 0.
\end{eqnarray*}
As $2ry_{0,0} = 8r^2q_0$ and $q_1 = 4r^2 q_0$, the trend hypothesis follows.  Boundedness follows easily from the fact that $\max\{|d_{j,v}^{\pm}|\} \le 2$.

The initial conditions are rather trivial to verify: the initial error is $0$ for all but $|W_{0,v}(0)|$, where it is exactly $1 = o(n^{1-1/6})$.  We turn now to Condition $4$ of Lemma \ref{lem:demeth}.  Nearly all of the required inequalities were previously verified in Section \ref{sec:analytic}, except those which involve $(t-r(t))$ or $\theta_r$: it follows easily that $|\frac{d}{dt}[t-r(t)]|$ is bounded and $\frac{d^2}{dt^2}[t-r(t)] = -r''(t)$.

Furthermore, as $h_3(t)=h_4(t) = \theta_r'(t)$, it follows from \eqref{eq:hvals2} and the facts that $e^{4t^2} < \theta_y(t)$, $\theta_y(m/s) < n^{\ep}$, and $\ep < \frac{1}{40}$ that \[\int_0^{m/s} |h_j'(t)|\ dt \le \sup_{0 \le t \le m/s} |h_j(t)| = h_j(m/s) \ll n^{1/12}\] for $j=3,4$, completing the proof.

\section{Independence Number - Proof of Theorem \ref{thm:indepno}} \label{sec:independence}

To prove Theorem \ref{thm:indepno}, we use the general approach to bounding the independence number of the triangle-free process presented in \cite{B}.  Recall that edge $e_i$ is blue if and only if $e_i \in O_{0,i-1}$, and that Theorem \ref{thm:trajectory} yields that $O_{0,i}$ forms a substantial portion of $O_i$ for $i=0,\ldots,m$.

We therefore argue as follows: first, we show that if $I \subseteq [n]$ is large enough then $I$ contains roughly $q_0(t)|I|^2$ pairs that are ``almost" in $O_{0,i}$.  Then, we argue that, for $|I| = \gamma \sqrt{\log(n)\cdot n}$ with suitably large $\gamma$, a positive fraction of these pairs are actually in $O_{0,i}$.  It will then follow that the probability that $e_{i} \notin \bin{I}{2} \cap O_{0,i}$ for all $i$ is significantly smaller than the number of choices of $I$, and the union bound will yield the final result.

Let $\rho = \frac{1}{5}$, $\beta = 3\mu$, and let $\gamma$ be a suitably large positive constant.  We note that, as a consequence of Corollary \ref{cor:maxdeg}, with high probability we have $\beta \sqrt{n \log(n)} > \Delta(G_i) + C n^{2\rho}\log^2(n)$ for any $C>0$ and all $i=0,\ldots,m$.

Let $k = \frac{\gamma-\beta}{2}\sqrt{n\log(n)}$.  For $A \subseteq [n]$, let \[L_{A}(i) = \left\{v \in [n] : |N_v(i) \cap A| \ge \frac{k}{n^{\rho}} \right\}\text{ and } L_A'(i) = \left\{v \in [n] : |N_v(i) \cap A| \ge \frac{k}{n^{\rho}}-1\right\},\] where $N_v(i) = N_{0,v}(i) \cup N_{1,v}(i)$.

We say a pair $uv \in \bin{A}{2}$ is \textbf{weakly open} with respect to $A$ at time $i^*$ if $uv \notin E_{i^*}$ and $Z_{uv}(i) \subseteq L_A(i)$ for $i=0,1,\ldots,i^*$.  A pair which is not weakly open is \textbf{weakly closed} (even if it forms an edge in the graph $G_i$).  Let $W_A(i)$ be the set of weakly open pairs in $A$ at time $i$.  Finally, let $\ca{H}_{i^*}$ be the event that the conclusions of Theorems \ref{thm:trajectory} and \ref{thm:regularity} and Lemma \ref{lem:bdchange} hold for $0 \le i \le i^*$.

Our first step is to show $|W_A|$ is of the ``right" size.

\begin{lemma}\label{lem:wopen}
Let $\lambda \ge 1$ be fixed.  With high probability, for all $A$ with $|A| = \lambda k$, \[|W_A(i)| = \lp q_0(t) \pm \frac{\theta_y(t)}{n^{1/20}}\rp |A|^2.\]
\end{lemma}

\begin{proof}
We apply Lemma \ref{lem:demeth} with the following values: $s = n^{3/2}$, $\epsilon = \frac{1}{40}$, $c$ and $C$ are the same as in Section \ref{sec:analytic}, $s_1 = n^{1/20}$, $k_1 = \lambda k$, $U_{1,A} = |W_A|$, $u_{1,A} = q_0$, $S_1 = (\lambda k)^2$, $\theta_1 = \theta_y$.  $\ca{H}_i$ is as above, and we let $\delta_w = \frac{\theta_y}{n^{1/20}} = n^{7/60}\delta_y$.  Conditions $3$ and $4$ of Lemma \ref{lem:demeth} are straightforward to verify, so we only argue the trend and boundedness hypotheses.

Let $|W_A(i+1)|-|W_A(i)| = W_A^+(i)-W_A^-(i)$: clearly $W_A^+(i)=0$.  As a weakly open pair $uv \in \bin{A}{2}$ becomes weakly closed only when $e_{i+1}$ connects $uv$ to $Y_{uv}(i)\setminus L_A'(i)$ or $e_{i+1}=uv$, it follows that $|W_A^-(i)| \le 2\frac{k}{n^{\rho}}+1$ which suffices as $n^{\rho} \gg n^{1/8}$.  It remains to verify the trend hypothesis for $W^-_A$, for which we will appeal to the following simple claim which follows from the bound on the codegrees.

\begin{claim}\label{clm:L_A}
Let $\lambda \ge 1$, and $A \in \bin{[n]}{\lambda k}$.  Conditioned on $\ca{H}_i$, $|L_A'(i)| \le 2\lambda n^{\rho}$.
\end{claim}

\noindent Therefore
\begin{eqnarray*}
\ev{W_A^-(i) | \gd} &=& \sum_{uv \in W_A(i)} \frac{|Y_{uv}| \pm 1 \pm 2\lambda n^{\rho}}{Q}\\
&\subseteq& \lp\frac{(q_0 \pm \delta_w)(y \pm (4\delta_y + \frac{3\lambda n^{\rho}}{\sqrt{n}}))}{q \pm 2\delta_q} \rp \frac{\lambda^2 k^2}{n^{3/2}}\\
&\subseteq& \lp \frac{q_0y}{q} \pm O(t+1)\delta_w \rp \frac{\lambda^2 k^2}{n^{3/2}},
\end{eqnarray*} and as $y = 8tq$, the trend hypothesis is verified.\\
\end{proof}

\noindent Now, suppose $\ca{H}_i$ holds and fix an $I \subseteq [n]$ with $|I| = \gamma \sqrt{n \cdot \log(n)}$.

\begin{claim}
\begin{equation}\label{eq:openpairsinI}
\left|O_{0,i} \cap \bin{I}{2}\right| \ge \frac{q(t)}{6} |I|^2.
\end{equation}
\end{claim}

\begin{proof}
We begin by casting aside a small portion of the vertices in $I$ as follows: let \[I' = (I \cap L_I(i)) \cup \{v \in I : |N_v(i) \cap L_I(i)| \ge 2\}.\]

It follows from Claim \ref{clm:L_A} and Lemma \ref{lem:bdchange} that
\begin{equation}\label{eq:I'}
|I'| \le 2\lp \frac{2\gamma}{\gamma -\beta}\rp n^{\rho} +  \lp 2 \cdot \frac{2\gamma}{\gamma-\beta} \cdot n^{\rho}\rp^2 \log^2(n) \le 65 n^{2\rho}\log^2(n).
\end{equation}

\noindent We then partition $I$ into sets $A$, $B$, $C$, with $|A| = |B| = k$ as follows: let $\widehat{L} \subseteq L_I(i)$ be a maximal set of vertices such that \[\left|\lp \bigcup_{v \in \widehat{L}} N_v(i)\rp \cap (I \setminus I') \right| \le k.\]  If $\widehat{L} = L_I$, then let $A \subseteq I$ be any $k$-element subset of $I \setminus I'$ containing $\lp \bigcup_{v \in \widehat{L}} N_v(i)\rp \cap (I\setminus I')$, and let $B$ be any $k$-element subset of $I \setminus (A \cup I')$.  Otherwise, let $v' \in L_I \setminus \hat{L}$ be arbitrary, and let $A$ be any $k$-element subset of $\bigcup_{v \in \widehat{L} \cup \{v'\}} N_v(i) \cap (I \setminus I')$ and $B$ be any $k$-element subset of $I \setminus (I' \cup \bigcup_{v \in \widehat{L} \cup \{v'\}} N_v(i))$.  It follows from our choice of $\beta$, Corollary \ref{cor:maxdeg}, and \eqref{eq:I'} that such a $B$ exists.  Finally, let $C = I \setminus (A \cup B)$.

By our choice of $A$ and $B$, we have that for all $v \in L_I$, at least one of $N_v(i) \cap A$ and $N_v(i) \cap B$ is empty.  This implies the final facts we need:
$W_I \cap (A \times B) \subseteq O_{0,i}$, and $W_I \setminus (A \times B) \subseteq W_{A \cup C} \cup W_{B \cup C}$.  Provided $\gamma$ is chosen sufficiently larger than $\beta$, it follows from two applications of Lemma \ref{lem:wopen} that \[|W_I \cap (A \times B)|\ge \lp q_0 - 2\delta_w \rp \frac{|I|^2}{3} > \frac{q_0}{4} |I|^2,\] and \eqref{eq:openpairsinI} follows by noting $q_0 \ge \frac{2}{3} q$.
\end{proof}

Following the remainder of the argument from \cite{B}, as $\ca{H}_i$ holds, $Q(i) \le (3/2)q(t)n^2$, so \[\pr{e_{i+1} \in O_{0,i} \cap \bin{I}{2}} \ge \frac{\gamma^2 \log(n)}{10n},\] and consequently the probability that $e_{i+1} \notin O_{0,i} \cap \bin{I}{2}$ for $i=0,\ldots,m-1$, conditioned on $\ca{H}_m$, is at most
\[\lp 1 - \frac{\gamma^2 \log(n)}{9n} \rp^{\mu \sqrt{\log(n)}n^{3/2}} \le \exp\lp-\frac{\gamma^2 \mu}{9} \log^{3/2}n \cdot \sqrt{n}\rp,\] while \[\bin{n}{\gamma \sqrt{n \log(n)}} \le \exp\lp \frac{\gamma}{2} \log^{3/2} n \cdot \sqrt{n} \rp.\]  By the union bound, with high probability all such $I$ contain at least one blue edge in $G_m$.

\section{Concluding Remarks}

An immediate consequence of Theorem \ref{thm:regularity} is that, w.h.p., $|E_{0,m}| = o(m)$, and therefore $e(G_{blue,m}) = (2/3 + o(1)) e(G_m)$.  Despite the fact that we believe $m$ to be only a small fraction of the final size, we believe this relationship holds for the final graph as well, which we state as a conjecture.

\begin{conjecture}
\[e(G_{blue}) = \lp \frac{2}{3} + o(1) \rp e(G_M).\]
\end{conjecture}

\section*{Acknowledgement} The author thanks Tom Bohman for suggesting the problem as well as many helpful comments on earlier drafts.

\end{document}